\documentclass[12pt,a4paper,reqno]{amsart}

\usepackage[margin=1in,footskip=0.25in]{geometry}
\usepackage{amsmath, amsthm, amssymb}
\usepackage{graphicx}
\usepackage{mathrsfs}

\usepackage{float}

\usepackage{booktabs}

\usepackage{hyperref}
\usepackage{xcolor}

\newtheorem{theorem}{Theorem}

\date{}

\allowdisplaybreaks
\begin{document}

\title[Crossing limit cycles for a family of isochronous centers]
{Crossing limit cycles of discontinuous piecewise differential systems with Pleshkan's isochronous centers}

\author{Sonia Isabel Renteria Alva$^1$}
\address{$^1$ Instituto de Matemática Pura e Aplicada, Estrada Dona Castorina 110, Jardim Botânico, Rio de Janeiro, 22460-320, Brazil}
\email{sonia.alva@impa.br}

\author{Pedro Iván Suárez Navarro$^2$}
\address{$^2$ Instituto de Matemática Pura e Aplicada,  Estrada Dona Castorina 110, Jardim Botânico, Rio de Janeiro, 22460-320, Brazil}
\email{ivan.suarez@impa.br}



\begin{abstract}
In recent decades, piecewise linear differential systems have attracted considerable attention due to their ability to describe a wide range of phenomena. A central problem, as in the theory of general planar differential systems, is to determine the existence and the maximal number of crossing limit cycles. However, deriving sharp upper bounds for this quantity remains a highly challenging problem. In this work we study crossing limit cycles in planar discontinuous piecewise differential systems separated by a straight line, where each subsystem is either a linear center or a cubic isochronous center with homogeneous nonlinearities. Within this setting, we consider all possible combinations arising from these families, leading to fifteen distinct classes of piecewise systems. Using the existence of first integrals, we reduce the detection of crossing limit cycles to algebraic closing conditions on the discontinuity set, which allows for a systematic and unified analysis across all configurations. As a consequence, we establish explicit upper bounds for the number of crossing limit cycles in all cases except for three configurations that remain open. In addition, we construct examples exhibiting three crossing limit cycles in every class, providing a nontrivial uniform lower bound. Our results extend and complement earlier work in the literature by including previously unstudied configurations and improving some known bounds, thereby providing a comprehensive description of the number of crossing limit cycles within this class of systems
\end{abstract}

\maketitle

\section{Introduction and statement of the main results}
The study of limit cycles in planar differential systems is a classical and central topic in dynamical systems, dating back to the pioneering works of Poincaré and closely related to the second part of Hilbert’s 16th problem. In this context, one considers planar systems defined by polynomial vector fields of degree $n$, where each component is given by a polynomial function in the plane. The second part of Hilbert’s 16th problem concerns determining the maximum number of limit cycles that such systems can exhibit for all possible choices of the defining polynomials; for more details see~\cite{ilyashenko2002centennial, zhang1992qualitative}. Recall that a \emph{limit cycle} is an isolated periodic orbit, and its existence play a fundamental role in understanding the qualitative behavior of nonlinear systems, with numerous applications in physics, engineering, and applied sciences; see~\cite{bernardo2008piecewise, simpson2010bifurcations}.

In recent years, increasing attention has been devoted to discontinuous piecewise differential systems, motivated by their relevance in modeling phenomena with abrupt transitions, such as switching systems, control processes, mechanical impacts, and electrical circuits. In this context, one of the main difficulties is the analysis of periodic solutions, particularly the determination of the number and configuration of limit cycles.

When dealing with planar discontinuous systems separated by a straight line, two types of limit cycles may arise: sliding and crossing limit cycles. In this work we focus on crossing limit cycles, namely, periodic orbits that intersect the discontinuity set at crossing points. The study of such solutions is especially challenging due to the interaction between the dynamics defined on each side of the discontinuity.

A powerful approach to this problem relies on the existence of first integrals. When each subsystem is integrable, the detection of crossing limit cycles can be reduced to solving algebraic closing conditions on the discontinuity line. This strategy has been successfully applied in several works, particularly for systems formed by linear and cubic isochronous centers, where explicit upper bounds for the number of crossing limit cycles have been obtained.

Isochronous centers, and in particular those classified by Pleshkan \cite{pleshkan1969new}, provide a natural and tractable framework for the study of nonlinear planar differential systems. These systems are characterized by the property that all periodic orbits surrounding the center have the same period, and, in many cases, they admit explicit first integrals with a relatively simple algebraic structure; see, for instance, \cite{loud1964behavior,chavarriga1999survey}. This feature makes them especially suitable for the analysis of discontinuous piecewise differential systems, where the detection of crossing limit cycles can be reduced to solving algebraic closing conditions on the discontinuity set. In recent years, this approach has been successfully applied to several classes of piecewise systems; see, for example, \cite{llibre2018piecewise,llibre2022crossing,huan2019limit,baymout2024limit,he2024limit}. In particular, for systems formed by linear and cubic isochronous centers, explicit upper bounds for the number of crossing limit cycles have been obtained in \cite{buzzi2022crossing}.
 

In this paper we study planar piecewise differential systems of the form
\begin{eqnarray}\label{eqmainOrig}
(\dot x, \dot y)=\begin{cases}
F^{+}(x,y)=(X^{+}(x,y), Y^{+}(x,y)), & \text{if } (x,y)\in \Sigma^{+},\\
F^{-}(x,y)=(X^{-}(x,y), Y^{-}(x,y)), & \text{if } (x,y)\in \Sigma^{-},
\end{cases}
\end{eqnarray}
where the separation curve is the straight line
$\Sigma = \Sigma^{+} \cap \Sigma^{-}$ where
\begin{eqnarray*}
\Sigma^{+} = \lbrace (x,y): x\geq 0\rbrace, \text{ and }
\Sigma^{-} = \lbrace (x,y): x\leq 0 \rbrace .
\end{eqnarray*}
We observe that differential system~\eqref{eqmainOrig} is bi-valued along the discontinuity line $\Sigma=\{(x,y): x=0\}$, in the sense that it admits two values, $F^+(0,y)$ and $F^-(0,y)$, on $\Sigma$.
According to \cite{filippov2013differential}, a point $p$ on the discontinuity set is a \textit{crossing point} if $X^{-}(p)X^{+}(p) > 0$. A periodic solution of system~\eqref{eqmainOrig} is said to be of crossing type if it intersects the discontinuity curve exactly twice at crossing points, and it is called a \textit{crossing limit cycle} if it is isolated among such solutions. For simplicity, throughout this paper we use the term \emph{limit cycle} to refer to crossing limit cycles.

Consider the polynomial differential systems of the form
\begin{equation}\label{sistemacubico}
\begin{array}{l}
\dfrac{dx}{dt}=\dot x = -y + a_{30}x^3 +a_{21}x^2 y+ a_{12}xy^2 +a_{03}y^3, \vspace{0.2cm} \\
\dfrac{dy}{dt}=\dot y = x + b_{30}x^3 +b_{21}x^2 y+ b_{12}xy^2 +b_{03}y^3.
\end{array}
\end{equation}

Pleshkan in \cite{pleshkan1969new} classify which of these differential systems have an isochronous center at the origin of coordinates. Thus a cubic system \eqref{sistemacubico} has an isochronous center at the origin if and only if the system can be transformed to one of the following four differential systems
\begin{equation}\label{sistemas1s2s3s4}
\begin{aligned}
&(\mathtt{S}_1): \; \left.
\begin{array}{l}
\dot{x}=-y + x^3 - x y^2, \\
\dot{y}=x + x^2y - y^3,
\end{array} \right. \; \; \; \; \; \; \; 
(\mathtt{S}_2): \; \left. 
\begin{array}{l}
\dot{x}=-y + x^3 - 3 x y^2, \\
\dot{y}=x + 3x^2y - y^3, 
\end{array} \right.  \\
& (\mathtt{S}_3): \; \left.
\begin{array}{l}
\dot{x}=-y + 3x^2y,  \\
\dot{y}=x -2x^3 +9 x y^2,
\end{array} \right. \; \; \; \; \; \;
(\mathtt{S}_4): \; 
\begin{array}{l}
\dot{x}=-y - 3x^2y, \\
\dot{y}=x + 2x^3-9xy^2.
\end{array} 
\end{aligned}
\end{equation}
doing a linear change of coordinates and a rescaling of time.

The first integrals of the differential systems \eqref{sistemas1s2s3s4} can be founded in \cite{chavarriga1999survey}, and they are
\begin{equation*}\label{intprimeiras1s2s3s4}
\begin{aligned}
&(\mathtt{S}_1): \widetilde{H_1}(x,y)=\dfrac{x^2+y^2}{1+2xy}, \; \; \; \; \; \; \; \; \; \; \; \; \; \; \; \;  \; \; \; \;  \; \; \; 
(\mathtt{S}_2): \widetilde{H_2}(x,y)=\dfrac{(x^2+y^2)^2} {1+4xy},\\
&(\mathtt{S}_3): \widetilde{H_3}(x,y)=\dfrac{x^2+y^2-4x^4+4x^6}{(-1+3x^2)^3}, \; \; \; \; \; 
(\mathtt{S}_4): \widetilde{H_4}(x,y)=\dfrac{x^2+y^2+4x^4+4x^6}{(1+3x^2)^3}.
\end{aligned}
\end{equation*}

Llibre and Teixeira proved in \cite{llibre2018piecewise} that after doing an affine transformation and a rescaling of the independent variable any linear center can be written into the form
\begin{equation}\label{sistemalinear}
\begin{aligned}  
(\mathtt{L}_c): \dot{x}= - A x -\dfrac{4 A^2 + \omega^2}{4 D} y + B, \quad 
\dot{y}=D x + A y + C,
\end{aligned}
\end{equation}
where $A, B, C,  D,  \omega$ are real numbers with $D, \omega>0$. This system has first integral
\begin{equation*}
H_L(x,y) = 4 ( D x + A y)^2 + 8 D (C x - B y) + y^2 \omega^2.
\end{equation*}

We investigate discontinuous piecewise differential systems separated by a straight line, where each subsystem, after an affine change of variables, reduces to one of the canonical forms $(\mathrm{Lc})$, $(\mathtt{S}_1)$, $(\mathtt{S}_2)$, $(\mathtt{S}_3)$, or $(\mathtt{S}_4)$. We consider the framework introduced in \cite{buzzi2022crossing}, where systems formed by combinations of $(\mathrm{Lc})$, $(\mathtt{S}_1)$, and $(\mathtt{S}_2)$ were analyzed using first integrals to reduce the problem to algebraic closing conditions on the discontinuity line, yielding upper bounds and examples of crossing limit cycles.

In the present work, we extend this analysis by incorporating all remaining combinations involving $(\mathtt{S}_3)$ and $(\mathtt{S}_4)$, thereby completing the study within this class. For most configurations, we determine explicit upper bounds and provide examples exhibiting at least three crossing limit cycles. In some cases where the exact bound remains open, we construct examples showing the existence of multiple limit cycles. In addition, we refine some previously reported bounds; for instance, in one configuration where an upper bound of nine was proposed, we prove that the correct bound is eight and provide examples with three crossing limit cycles. We also obtain examples attaining the maximal number of limit cycles in cases where previously only partial results were available.

Our first main result provides upper bounds for the number of crossing limit cycles that can appear in planar discontinuous piecewise differential systems separated by a straight line, where each subsystem belongs, after an affine change of variables, to one of the families $(\mathtt{L}_c)$, $(\mathtt{S}_1)$, $(\mathtt{S}_2)$, $(\mathtt{S}_3)$, or $(\mathtt{S}_4)$.
\begin{theorem}\label{Thm1}
Consider planar discontinuous piecewise differential systems separated by the straight line $x=0$ and formed by two subsystems which, after an affine change of variables, belong to the classes $(\mathtt{L}_c)$, $(\mathtt{S}_1)$, $(\mathtt{S}_2)$, $(\mathtt{S}_3)$, or $(\mathtt{S}_4)$. Then the maximum number of crossing limit cycles satisfies the following bounds:
\begin{itemize}
\item[(i)] at most three for systems of type $(\mathtt{L}_c)$ –$(\mathtt{S}_3)$, and there exist systems of this type with exactly three limit cycles, see Fig.~\ref{fig-HL-S3};
\item[(ii)] at most three for systems of type $(\mathtt{L}_c)$–$(\mathtt{S}_4)$, and there exist systems of this type with exactly three limit cycles, see Fig.~\ref{fig-HL-S4};
\item[(iii)] at most five for systems of type $(\mathtt{S}_3)$–$(\mathtt{S}_1)$, and there exist systems of this type with three limit cycles, see Fig.~\ref{fig-S1-S3};
\item[(iv)] at most five for systems of type $(\mathtt{S}_4)$–$(\mathtt{S}_1)$, and there exist systems of this type with three limit cycles, see Fig.~\ref{fig-S1-S4};
\item[(v)] at most eight for systems of type $(\mathtt{S}_2)$–$(\mathtt{S}_2)$, and there exist systems of this type with three limit cycles, see Fig.~\ref{fig-S2-S2};
\item[(vi)] at most thirteen for systems of type $(\mathtt{S}_3)$–$(\mathtt{S}_2)$, and there exist systems of this type with three limit cycles, see Fig.~\ref{fig-S2-S3};
\item[(vii)] at most thirteen for systems of type $(\mathtt{S}_4)$–$(\mathtt{S}_2)$, and there exist systems of this type with three limit cycles, see Fig.~\ref{fig-S2-S4}.
\end{itemize}
\end{theorem}

The second main result shows the existence of multiple crossing limit cycles for configurations not covered by the previous theorem.

\begin{theorem}\label{Thm2}
The following statements hold:
\begin{itemize}
\item[(i)] There exist piecewise differential systems formed by $(\mathtt{S}_1)$ and $(\mathtt{S}_2)$ with three crossing limit cycles, see Fig.~\ref{fig-S1-S2};
\item[(ii)] There exist piecewise differential systems formed by $(\mathtt{S}_3)$ and $(\mathtt{S}_3)$ with three crossing limit cycles, see Fig.~\ref{fig-S3-S3};
\item[(iii)] There exist piecewise differential systems formed by $(\mathtt{S}_3)$ and $(\mathtt{S}_4)$ with three crossing limit cycles, see Fig.~\ref{fig-S3-S4};
\item[(iv)] There exist piecewise differential systems formed by $(\mathtt{S}_4)$ and $(\mathtt{S}_4)$ with three crossing limit cycles, see Fig.~\ref{fig-S4-S4}.
\end{itemize}
\end{theorem}

Table~\ref{TableLCy} summarizes the results of Theorems~\ref{Thm1} and~\ref{Thm2}. The values in parentheses denote the maximal number of known crossing limit cycles for each class, and are realized by examples constructed in both theorems.

\begin{table}[h]
\caption{Upper bounds for the number of crossing limit cycles obtained in Theorem~\ref{Thm1} and \ref{Thm2}.}
\label{TableLCy}
\begin{tabular}{l | c | c | c | c | c }
\toprule
  & $\mathtt{L}_c$  & $\mathtt{S}_1$   & $\mathtt{S}_2$  &  $\mathtt{S}_3$ &  $\mathtt{S}_4$ \\
  \hline
  $\mathtt{L}_c$  & 0 & $1~(1)$  & $2~(2)$ & $\textcolor{blue}{3~(3)}$ & \textcolor{blue}{3~(3)} \\
    \hline
  $\mathtt{S}_1$   & $1~(1)$ & $1~(1)$  & $3~\textcolor{blue}{(3)}$ & $\textcolor{blue}{5~(3)}$ & $\textcolor{blue}{5~(3)}$ \\
   \hline
  $\mathtt{S}_2$   & $2~(2)$ & $3~\textcolor{blue}{(3)}$  & $\textcolor{blue}{8~(3)}$ & $\textcolor{blue}{13~(3)}$ & $\textcolor{blue}{13~(3)}$ \\
   \hline
  $\mathtt{S}_3$   & $\textcolor{blue}{3~(3)}$ & $\textcolor{blue}{5~(3)}$ & $\textcolor{blue}{13~(3)}$  & $\textcolor{blue}{~(3)}$  & $\textcolor{blue}{~(3)}$  \\
  \hline
  $\mathtt{S}_4$   & $\textcolor{blue}{3~(3)}$  & $\textcolor{blue}{5~(3)}$ & $\textcolor{blue}{13~(3)}$  & $\textcolor{blue}{ ~(3)}$  & $\textcolor{blue}{~(3)}$ \\
  \bottomrule
\end{tabular}

\end{table}

The values obtained in this work are highlighted in blue in Table~\ref{TableLCy}, while the remaining ones correspond to those reported in \cite{buzzi2022crossing}; see the corresponding table therein for comparison.

In Section~\ref{sec2:changevar}, we present the families of cubic isochronous centers after an affine change of variables and describe the expressions that will be used to derive the closing conditions. Section~\ref{sect:03} is devoted to the proof of Theorem~\ref{Thm1}, where we analyze all possible configurations and establish the corresponding upper bounds for the number of crossing limit cycles. In Section~\ref{sect:04}, we prove Theorem~\ref{Thm2} by constructing explicit systems exhibiting three crossing limit cycles.

\section{Cubic isochronous centers after an affine change of variables}
\label{sec2:changevar}
In this section, we give the expression of the cubic isochronous centers with homogeneous nonlinearities  ($\mathtt{S}_1$), ($\mathtt{S}_2$), ($\mathtt{S}_3$),($\mathtt{S}_4$), as well as to their first integrals after the general affine change of
variables 
 $(x, y) \rightarrow (a_1 x +b_1 y + c_1, \alpha_1 x + \beta_1 y + \gamma_1)$,
with $b_1 \alpha_1 - a_1 \beta_1\neq 0$. 

The differential system ($\mathtt{S}_1$) becomes

\begin{equation*}
\begin{aligned}
\dot{x}=& \frac{1}{b_1 \alpha_1 - a_1 \beta_1}
\Big( -\beta_1 c_1^3 + b_1 \gamma_1 c_1^2 + \beta_1 \gamma_1^2 c_1 + b_1 c_1 - b_1 \gamma_1^3 - x^3 (a_1^2 - \alpha_1^2)(a_1 \beta_1 - b_1 \alpha_1) \\
& + x y^2 (b_1 \alpha_1 - a_1 \beta_1)(b_1^2 - \beta_1^2) - 2 x^2 y (b_1 \beta_1 a_1^2 - \alpha_1 \beta_1^2 a_1 - a_1 b_1^2 \alpha_1 + b_1 \alpha_1^2 \beta_1) + \beta_1 \gamma_1 \\
& + y^2 (b_1^2 - \beta_1^2)(b_1 \gamma_1 - c_1 \beta_1) + x^2 (-3 c_1 \beta_1 a_1^2 + b_1 \gamma_1 a_1^2 + 2 b_1 c_1 \alpha_1 a_1 + 2 \alpha_1 \beta_1 \gamma_1 a_1 + c_1 \\
& \alpha_1^2 \beta_1 - 3 b_1 \alpha_1^2 \gamma_1) + 2 x y (c_1 \alpha_1 b_1^2 + a_1 \gamma_1 b_1^2 - 2 a_1 c_1 \beta_1 b_1 - 2 \alpha_1 \beta_1 \gamma_1 b_1 + c_1 \alpha_1 \beta_1^2 + a_1 \beta_1^2 \\
& \gamma_1) + x (b_1 \alpha_1 c_1^2 - 3 a_1 \beta_1 c_1^2 + 2 a_1 b_1 \gamma_1 c_1 + 2 \alpha_1 \beta_1 \gamma_1 c_1 - 3 b_1 \alpha_1 \gamma_1^2 + a_1 \beta_1 \gamma_1^2 + a_1 b_1 + \alpha_1 \\
& \beta_1) + y (2 c_1 \gamma_1 b_1^2 + b_1^2 - 2 \beta_1 \gamma_1^2 b_1 - 2 c_1^2 \beta_1 b_1 + \beta_1^2 + 2 c_1 \beta_1^2 \gamma_1) \Big),
\end{aligned}
\end{equation*}

\begin{equation}\label{syst:S1}
\begin{aligned}
\dot{y}=&\frac{1}{b_1 \alpha_1 - a_1 \beta_1}
\Big( -\alpha_1 c_1^3 + a_1 \gamma_1 c_1^2 + \alpha_1 \gamma_1^2 c_1 + a_1 c_1 - a_1 \gamma_1^3 + x^2 y (a_1^2 - \alpha_1^2)(a_1 \beta_1 - b_1 \alpha_1) \\
& - y^3 (b_1 \alpha_1 - a_1 \beta_1)(b_1^2 - \beta_1^2) + 2 x y^2 (b_1 \beta_1 a_1^2 - \alpha_1 \beta_1^2 a_1 - a_1 b_1^2 \alpha_1 + b_1 \alpha_1^2 \beta_1) + \alpha_1 \gamma_1 \\
& + x^2 (a_1^2 - \alpha_1^2)(a_1 \gamma_1 - c_1 \alpha_1) + 2 x y (c_1 \beta_1 a_1^2 + b_1 \gamma_1 a_1^2 - 2 b_1 c_1 \alpha_1 a_1 - 2 \alpha_1 \beta_1 \gamma_1 a_1 + c_1 \\
& \alpha_1^2 \beta_1 + b_1 \alpha_1^2 \gamma_1) + y^2 (-3 c_1 \alpha_1 b_1^2 + a_1 \gamma_1 b_1^2 + 2 a_1 c_1 \beta_1 b_1 + 2 \alpha_1 \beta_1 \gamma_1 b_1 + c_1 \alpha_1 \beta_1^2 - 3 a_1 \\
& \beta_1^2 \gamma_1) + x (2 c_1 \gamma_1 a_1^2 + a_1^2 - 2 \alpha_1 \gamma_1^2 a_1 - 2 c_1^2 \alpha_1 a_1 + \alpha_1^2 + 2 c_1 \alpha_1^2 \gamma_1) + y (-3 b_1 \alpha_1 c_1^2 + a_1 \\
& \beta_1 c_1^2 + 2 a_1 b_1 \gamma_1 c_1 + 2 \alpha_1 \beta_1 \gamma_1 c_1 + b_1 \alpha_1 \gamma_1^2 - 3 a_1 \beta_1 \gamma_1^2 + a_1 b_1 + \alpha_1 \beta_1) \Big),
\end{aligned}
\end{equation}
with the first integral 
\begin{equation*}\label{sistemamud}
\begin{aligned}
  H_1(x,y)
&=
\frac{(a_1 x + b_1 y + c_1)^2+(\alpha_1 x + \beta_1 y + \gamma_1)^2}
{1+2\,(a_1 x + b_1 y + c_1)(\alpha_1 x + \beta_1 y + \gamma_1)}.
\end{aligned}
\end{equation*}

The differential system ($\mathtt{S}_2$) becomes

\begin{equation*}
\begin{aligned}
\dot{x}=&\frac{1}{b_1 \alpha_1 - a_1 \beta_1}
\Big( -\beta_1 c_1^3 + 3 b_1 \gamma_1 c_1^2 + 3 \beta_1 \gamma_1^2 c_1 + b_1 c_1 - b_1 \gamma_1^3 + x^3 (-\beta_1 a_1^3 + 3 b_1 \alpha_1 a_1^2 + 3 \\
& \alpha_1^2 \beta_1 a_1 - b_1 \alpha_1^3) + 3 x y^2 (b_1 \alpha_1 + a_1 \beta_1)(b_1^2 + \beta_1^2) + 6 x^2 y (a_1 \alpha_1 b_1^2 + a_1 \alpha_1 \beta_1^2) + 2 y^3 (\beta_1 b_1^3 \\
& + \beta_1^3 b_1) + \beta_1 \gamma_1 + 6 x y (b_1^2 + \beta_1^2)(c_1 \alpha_1 + a_1 \gamma_1) + 3 y^2 (b_1^2 + \beta_1^2)(c_1 \beta_1 + b_1 \gamma_1) + y (b_1^2 \\
& + \beta_1^2)(6 c_1 \gamma_1 + 1) + 3 x^2 (-c_1 \beta_1 a_1^2 + b_1 \gamma_1 a_1^2 + 2 b_1 c_1 \alpha_1 a_1 + 2 \alpha_1 \beta_1 \gamma_1 a_1 + c_1 \alpha_1^2 \beta_1 - b_1 \\
& \alpha_1^2 \gamma_1) + x (3 b_1 \alpha_1 c_1^2 - 3 a_1 \beta_1 c_1^2 + 6 a_1 b_1 \gamma_1 c_1 + 6 \alpha_1 \beta_1 \gamma_1 c_1 - 3 b_1 \alpha_1 \gamma_1^2 + 3 a_1 \beta_1 \gamma_1^2 + a_1 b_1 \\
& + \alpha_1 \beta_1) \Big),
\end{aligned}
\end{equation*}

\begin{equation}\label{syst:S2}
\begin{aligned}
\dot{y}=&\frac{1}{b_1 \alpha_1 - a_1 \beta_1}
\Big( -\alpha_1 c_1^3 + 3 a_1 \gamma_1 c_1^2 + 3 \alpha_1 \gamma_1^2 c_1 + a_1 c_1 - a_1 \gamma_1^3 + 2 x^3 (\alpha_1 a_1^3 + \alpha_1^3 a_1) + 3 x^2 \\
& y (a_1^2 + \alpha_1^2)(b_1 \alpha_1 + a_1 \beta_1) + 6 x y^2 (b_1 \beta_1 a_1^2 + b_1 \alpha_1^2 \beta_1) + y^3 (3 a_1 \beta_1 b_1^2 + 3 \alpha_1 \beta_1^2 b_1 - a_1 \beta_1^3 \\
& - b_1^3 \alpha_1) + \alpha_1 \gamma_1 + 3 x^2 (a_1^2 + \alpha_1^2)(c_1 \alpha_1 + a_1 \gamma_1) + 6 x y (a_1^2 + \alpha_1^2)(c_1 \beta_1 + b_1 \gamma_1) + x (a_1^2 \\
& + \alpha_1^2)(6 c_1 \gamma_1 + 1) - 3 y^2 (c_1 \alpha_1 b_1^2 - a_1 \gamma_1 b_1^2 - 2 a_1 c_1 \beta_1 b_1 - 2 \alpha_1 \beta_1 \gamma_1 b_1 - c_1 \alpha_1 \beta_1^2 + a_1 \\
& \beta_1^2 \gamma_1) + y (-3 b_1 \alpha_1 c_1^2 + 3 a_1 \beta_1 c_1^2 + 6 a_1 b_1 \gamma_1 c_1 + 6 \alpha_1 \beta_1 \gamma_1 c_1 + 3 b_1 \alpha_1 \gamma_1^2 - 3 a_1 \beta_1 \gamma_1^2 + a_1 \\
& b_1 + \alpha_1 \beta_1) \Big),
\end{aligned}
\end{equation}

with the first integral 
\begin{equation*}\label{sistemamud}
\begin{aligned}
H_2(x,y)
&=
\frac{\big((a_1 x + b_1 y + c_1)^2+(\alpha_1 x + \beta_1 y + \gamma_1)^2\big)^2}
{1+4\,(a_1 x + b_1 y + c_1)(\alpha_1 x + \beta_1 y + \gamma_1)}.
\end{aligned}
\end{equation*}

The differential system ($\mathtt{S}_3$) becomes
\begin{equation*}
\begin{aligned}
\dot{x} =& \frac{1}{b_1 \alpha_1 - a_1 \beta_1} 
\Big( -2 b_1 c_1^3 - 3 \beta_1 \gamma_1 c_1^2 + 9 b_1 \gamma_1^2 c_1 + b_1 c_1 + x^3 (-2 b_1 a_1^3 - 3 \alpha_1 \beta_1 a_1^2 + 9 b_1 \alpha_1^2 a_1) \\
& - 3 x^2 y (2 a_1^2 b_1^2 - 3 \alpha_1^2 b_1^2 - 4 a_1 \alpha_1 \beta_1 b_1 + a_1^2 \beta_1^2) - 3 x y^2 (2 a_1 b_1^3 - 5 \alpha_1 \beta_1 b_1^2 - a_1 \beta_1^2 b_1) - 2 y^3 \\
& (b_1^4 - 3 b_1^2 \beta_1^2) + \beta_1 \gamma_1 - 3 x^2 (2 b_1 c_1 a_1^2 + \beta_1 \gamma_1 a_1^2 + 2 c_1 \alpha_1 \beta_1 a_1 - 6 b_1 \alpha_1 \gamma_1 a_1 - 3 b_1 c_1 \alpha_1^2) - 6 \\
& x y (2 a_1 c_1 b_1^2 - 3 \alpha_1 \gamma_1 b_1^2 - 2 c_1 \alpha_1 \beta_1 b_1 - 2 a_1 \beta_1 \gamma_1 b_1 + a_1 c_1 \beta_1^2) - 3 y^2 (2 c_1 b_1^3 - 5 \beta_1 \gamma_1 b_1^2 - c_1 \\
& \beta_1^2 b_1) + x (-6 a_1 b_1 c_1^2 - 3 \alpha_1 \beta_1 c_1^2 + 18 b_1 \alpha_1 \gamma_1 c_1 - 6 a_1 \beta_1 \gamma_1 c_1 + 9 a_1 b_1 \gamma_1^2 + a_1 b_1 + \alpha_1 \beta_1) \\
& + y (-6 c_1^2 b_1^2 + 9 \gamma_1^2 b_1^2 + b_1^2 + 12 c_1 \beta_1 \gamma_1 b_1 - 3 c_1^2 \beta_1^2 + \beta_1^2) \Big),
\end{aligned}
\end{equation*}

\begin{equation}\label{syst:S3}
\begin{aligned}
\dot{y} =& \frac{1}{b_1 \alpha_1 - a_1 \beta_1} 
\Big( 2 a_1 c_1^3 + 3 \alpha_1 \gamma_1 c_1^2 - 9 a_1 \gamma_1^2 c_1 - a_1 c_1 + 2 x^3 (a_1^4 - 3 a_1^2 \alpha_1^2) + 3 x^2 y (2 b_1 a_1^3 - 5 \\
& \alpha_1 \beta_1 a_1^2 - b_1 \alpha_1^2 a_1) + b_1 y^3 (2 a_1 b_1^2 + 3 \alpha_1 \beta_1 b_1 - 9 a_1 \beta_1^2) + 3 x y^2 (2 a_1^2 b_1^2 + \alpha_1^2 b_1^2 - 4 a_1 \alpha_1 \beta_1 b_1 \\
& - 3 a_1^2 \beta_1^2) - \alpha_1 \gamma_1 + 3 x^2 (2 c_1 a_1^3 - 5 \alpha_1 \gamma_1 a_1^2 - c_1 \alpha_1^2 a_1) + 6 x y (2 b_1 c_1 a_1^2 - 3 \beta_1 \gamma_1 a_1^2 - 2 c_1 \alpha_1 \\
& \beta_1 a_1 - 2 b_1 \alpha_1 \gamma_1 a_1 + b_1 c_1 \alpha_1^2) + 3 y^2 (2 a_1 c_1 b_1^2 + \alpha_1 \gamma_1 b_1^2 + 2 c_1 \alpha_1 \beta_1 b_1 - 6 a_1 \beta_1 \gamma_1 b_1 - 3 a_1 c_1 \\
& \beta_1^2) + x (6 c_1^2 a_1^2 - 9 \gamma_1^2 a_1^2 - a_1^2 - 12 c_1 \alpha_1 \gamma_1 a_1 + 3 c_1^2 \alpha_1^2 - \alpha_1^2) + y (6 a_1 b_1 c_1^2 + 3 \alpha_1 \beta_1 c_1^2 + 6 b_1 \\
& \alpha_1 \gamma_1 c_1 - 18 a_1 \beta_1 \gamma_1 c_1 - 9 a_1 b_1 \gamma_1^2 - a_1 b_1 - \alpha_1 \beta_1) \Big),
\end{aligned}
\end{equation}

with the first integral 
\begin{equation*}\label{sistemamud}
\begin{aligned}
 H_3(x,y)
&=
\frac{(a_1 x + b_1 y + c_1)^2+(\alpha_1 x + \beta_1 y + \gamma_1)^2
-4(a_1 x + b_1 y + c_1)^4+4(a_1 x + b_1 y + c_1)^6}
{\big(-1+3(a_1 x + b_1 y + c_1)^2\big)^3}.
\end{aligned}
\end{equation*}

The differential system ($\mathtt{S}_4$) becomes

\begin{equation*}
\begin{aligned}
\dot{x} =& \frac{1}{b_1 \alpha_1 - a_1 \beta_1} 
\Big( 2 b_1 c_1^3 + 3 \beta_1 \gamma_1 c_1^2 - 9 b_1 \gamma_1^2 c_1 + b_1 c_1 + a_1 x^3 (2 b_1 a_1^2 + 3 \alpha_1 \beta_1 a_1 - 9 b_1 \alpha_1^2) + 3 \\
& x^2 y (2 a_1^2 b_1^2 - 3 \alpha_1^2 b_1^2 - 4 a_1 \alpha_1 \beta_1 b_1 + a_1^2 \beta_1^2) + 3 x y^2 (2 a_1 b_1^3 - 5 \alpha_1 \beta_1 b_1^2 - a_1 \beta_1^2 b_1) + 2 y^3 (b_1^4 \\
& - 3 b_1^2 \beta_1^2) + \beta_1 \gamma_1 + 3 x^2 (2 b_1 c_1 a_1^2 + \beta_1 \gamma_1 a_1^2 + 2 c_1 \alpha_1 \beta_1 a_1 - 6 b_1 \alpha_1 \gamma_1 a_1 - 3 b_1 c_1 \alpha_1^2) + 6 x y \\
& (2 a_1 c_1 b_1^2 - 3 \alpha_1 \gamma_1 b_1^2 - 2 c_1 \alpha_1 \beta_1 b_1 - 2 a_1 \beta_1 \gamma_1 b_1 + a_1 c_1 \beta_1^2) + 3 y^2 (2 c_1 b_1^3 - 5 \beta_1 \gamma_1 b_1^2 - c_1 \beta_1^2 \\
& b_1) + x (6 a_1 b_1 c_1^2 + 3 \alpha_1 \beta_1 c_1^2 - 18 b_1 \alpha_1 \gamma_1 c_1 + 6 a_1 \beta_1 \gamma_1 c_1 - 9 a_1 b_1 \gamma_1^2 + a_1 b_1 + \alpha_1 \beta_1) + y (6 \\
& c_1^2 b_1^2 - 9 \gamma_1^2 b_1^2 + b_1^2 - 12 c_1 \beta_1 \gamma_1 b_1 + 3 c_1^2 \beta_1^2 + \beta_1^2) \Big), 
\end{aligned}
\end{equation*}
\begin{equation}\label{syst:S4}
\begin{aligned}
\dot{y}=& \frac{1}{b_1 \alpha_1 - a_1 \beta_1}
\Big( -2 a_1 c_1^3 - 3 \alpha_1 \gamma_1 c_1^2 + 9 a_1 \gamma_1^2 c_1 - a_1 c_1 - 2 x^3 (a_1^4 - 3 a_1^2 \alpha_1^2) - 3 x^2 y (2 b_1 a_1^3 \\
& - 5 \alpha_1 \beta_1 a_1^2 - b_1 \alpha_1^2 a_1) + y^3 (-2 a_1 b_1^3 - 3 \alpha_1 \beta_1 b_1^2 + 9 a_1 \beta_1^2 b_1) - 3 x y^2 (2 a_1^2 b_1^2 + \alpha_1^2 b_1^2 - 4 a_1 \\
& \alpha_1 \beta_1 b_1 - 3 a_1^2 \beta_1^2) - \alpha_1 \gamma_1 - 3 x^2 (2 c_1 a_1^3 - 5 \alpha_1 \gamma_1 a_1^2 - c_1 \alpha_1^2 a_1) - 6 x y (2 b_1 c_1 a_1^2 - 3 \beta_1 \gamma_1 a_1^2 \\
& - 2 c_1 \alpha_1 \beta_1 a_1 - 2 b_1 \alpha_1 \gamma_1 a_1 + b_1 c_1 \alpha_1^2) - 3 y^2 (2 a_1 c_1 b_1^2 + \alpha_1 \gamma_1 b_1^2 + 2 c_1 \alpha_1 \beta_1 b_1 - 6 a_1 \beta_1 \gamma_1 \\
& b_1 - 3 a_1 c_1 \beta_1^2) + x (-6 c_1^2 a_1^2 + 9 \gamma_1^2 a_1^2 - a_1^2 + 12 c_1 \alpha_1 \gamma_1 a_1 - 3 c_1^2 \alpha_1^2 - \alpha_1^2) + y (-6 a_1 b_1 c_1^2 \\
& - 3 \alpha_1 \beta_1 c_1^2 - 6 b_1 \alpha_1 \gamma_1 c_1 + 18 a_1 \beta_1 \gamma_1 c_1 + 9 a_1 b_1 \gamma_1^2 - a_1 b_1 - \alpha_1 \beta_1) \Big),
\end{aligned}
\end{equation}

with the first integral 
\begin{equation*}\label{sistemamud}
\begin{aligned}
H_4(x,y)
&=
\frac{(a_1 x + b_1 y + c_1)^2+(\alpha_1 x + \beta_1 y + \gamma_1)^2
+4(a_1 x + b_1 y + c_1)^4+4(a_1 x + b_1 y + c_1)^6}
{\big(1+3(a_1 x + b_1 y + c_1)^2\big)^3}.
\end{aligned}
\end{equation*}
\section{Proof of Theorem $\ref{Thm1}$}
\label{sect:03}

\subsection*{\textbf{Proof of Theorem \ref{Thm1} for systems  $\mathtt{L}_c- \mathtt{S}_3$.}}
We consider in $\Sigma^-$ a cubic system \eqref{syst:S3}, that is, the cubic system $\mathtt{S}_3$ after an arbitrary affine change of variables,  with its first
integral $H_3(x, y)$; and in $\Sigma^+$,  a linear differential system $\mathtt{L}_c$ \eqref{sistemalinear} with its first integral $H_L(x, y)$.

If such discontinuous piecewise differential systems have a limit cycle intersecting the discontinuity straight line $x = 0$ in the two points $(0, y_1 )$ and $(0, y_2)$, then $y_1$ and $y_2$ must satisfy that
\begin{equation}\label{sist_Hl-H3}
\begin{aligned}
    H_L(0,y_1) - H_L(0,y_2) = 0, \text{ and }
    H_3(0,y_1) - H_3(0,y_2) = 0,
\end{aligned}
\end{equation}
or equivalently
\begin{align*}
\mathscr{E}_1 =& - (y_1 - y_2) \Big( -4 A^2 y_1 - 4A^2 y_2 + 8B\,D - \omega^2 y_1 - \omega^2 y_2 \Big) = 0, \\
\mathscr{E}_2 =& \frac{(y_1 - y_2)\mathscr{P}_3(y_1,y_2) }{\left(-1 + 3(c_1 + b_1 y_1)^2\right)^3 \left(-1 + 3(c_1 + b_1 y_2)^2\right)^3} = 0,
\end{align*}
where
\begin{align*}
\mathscr{P}_3(y_1,y_2) =& 9 b_1^8 y_1^2 y_2^2 (y_1 + y_2) (y_1^2 + y_2^2) + 18 b_1^7 c_1 y_1 y_2 (y_1^4 + 4 y_1^3 y_2 + 4 y_1^2 y_2^2 + 4 y_1 y_2^3 + y_2^4) \\
& + (-1 + 3 c_1^2)^3 \beta_1 ((y_1 + y_2) \beta_1 + 2 \gamma_1) + 2 b_1 c_1 (1 - 3 c_1^2)^2 (-1 + 2 c_1^2 + 9 y_1 y_2 \\
& \beta_1^2 - 9 \gamma_1^2) + b_1^2 (-1 + 3 c_1^2) ((1 - 21 c_1^2 + 42 c_1^4) (y_1 + y_2) + 18 (1 - 15 c_1^2) y_1 y_2 \\
& \beta_1 \gamma_1 - 9 (-1 + 15 c_1^2) (y_1 + y_2) \gamma_1^2) + b_1^4 ((y_1 + y_2) (72 c_1^2 (-1 + 5 c_1^2) y_1 y_2 \\
& + (4 - 69 c_1^2 + 135 c_1^4) y_2^2 + (-1 + 15 c_1^2) y_1^2 (-4 + 9 c_1^2 - 27 y_2^2 \beta_1^2)) - 54 \\
& (-1 + 15 c_1^2) y_1 y_2 (y_1^2 + y_1 y_2 + y_2^2) \beta_1 \gamma_1 - 27 (-1 + 15 c_1^2) (y_1 + y_2) (y_1^2 + y_2^2) \\
& \gamma_1^2) + b_1^6 ((y_1 + y_2) (108 c_1^2 y_1^3 y_2 + 108 c_1^2 y_1 y_2^3 + (-4 + 9 c_1^2) y_2^4 + y_1^4 (-4 + 9 c_1^2 \\
& - 27 y_2^2 \beta_1^2) + y_1^2 y_2^2 (-13 + 144 c_1^2 - 27 y_2^2 \beta_1^2)) - 54 y_1 y_2 (y_1^4 + y_1^3 y_2 + y_1^2 y_2^2 \\
& + y_1 y_2^3 + y_2^4) \beta_1 \gamma_1 - 27 (y_1 + y_2) (y_1^2 - y_1 y_2 + y_2^2) (y_1^2 + y_1 y_2 + y_2^2) \gamma_1^2) + 6 \\
& b_1^5 c_1 (y_2^4 (-4 + 9 c_1^2 - 27 \gamma_1^2) + y_1^4 (-4 + 9 c_1^2 - 27 (y_2 \beta_1 + \gamma_1)^2) - y_1^2 y_2^2 (13 \\
& - 84 c_1^2 + 27 (y_2 \beta_1 + \gamma_1)^2) - y_1^3 y_2 (7 - 54 c_1^2 + 27 (y_2 \beta_1 + \gamma_1)^2) - y_1 y_2^3 \\
& (7 - 54 c_1^2 + 27 \gamma_1 (2 y_2 \beta_1 + \gamma_1))) + 4 b_1^3 c_1 ((-1 + 5 c_1^2) y_2^2 (-4 + 9 c_1^2 - 27 \gamma_1^2) \\
& + (-1 + 5 c_1^2) y_1^2 (-4 + 9 c_1^2 - 27 (y_2 \beta_1 + \gamma_1)^2) + y_1 y_2 (4 + 99 c_1^4 + 27 \gamma_1 \\
& (2 y_2 \beta_1 + \gamma_1) - c_1^2 (47 + 135 \gamma_1 (2 y_2 \beta_1 + \gamma_1)))),
\end{align*}
where $\mathscr{P}_3(y_1,y_2)$ is a polynomial of degree $7$. Since $\left(-1 + 3(c_1 + b_1 y_1)^2\right) \left(-1 + 3(c_1 + b_1 y_2)^2\right) \neq 0$ and $y_1 < y_2$, we can remove these terms to solve the system, and we get the equivalent system
\begin{equation}\label{simplif_sist}
\begin{aligned}
\mathscr{\tilde{E}}_1 (y_1,y_2) =& -4 A^2 y_1 - 4A^2 y_2 + 8B\,DD - \omega^2 y_1 - \omega^2 y_2 = 0, \\
\mathscr{\tilde{E}}_2 (y_1,y_2) =& \mathscr{P}_3(y_1,y_2) = 0.
\end{aligned}
\end{equation}
System \eqref{simplif_sist} may admit a continuum of solutions $(y_1,y_2)$; in such a case, the associated periodic orbits are not isolated and thus are not limit cycles. Hence, we assume that \eqref{simplif_sist} has finitely many solutions. From $\mathscr{\tilde{E}}_1(y_1,y_2)$ it follows that
\begin{equation}\label{sol_y1}
\begin{aligned}
y_1 =& \frac{8 B\,D - y_2(4 A^2 + \omega^2)}{4 A^2 + \omega^2},
\end{aligned}
\end{equation}
with $4 A^2 + \omega^2 \neq 0$ because $\omega > 0$. So if we substitute \eqref{sol_y1} in $\mathscr{\tilde{E}}_2 (y_1,y_2) = 0$, then we get a polynomial $p_3(y_2)$ of degree $6$ in the variable $y_2$, and $p_3(y_2)$ has at most six roots. Therefore the system \eqref{simplif_sist} has at most six solutions, and consequently, the discontinuous piecewise differential system can have at most three limit cycles.

Now we shall prove that the discontinuous piecewise differential system $\mathtt{L}_c-\eqref{syst:S3}$ separated by the straight line $\Sigma:x=0$, having three limit cycles. 
In $\Sigma^+$, we consider the linear differential center 
\begin{align}\label{sist1-HL-S3}
    \dot{x} = \frac{9}{10} x - \frac{9}{10} y -\frac{7}{10}, \quad \dot{y} = x - \frac{9}{10} y - \frac{1}{2}, 
\end{align} 
with the first integral
\begin{equation*}
H_L(x,y) = 4 \left(x - \frac{9}{10} y \right)^2 + 8 \left(\frac{7}{10} y - \frac{1}{2} x \right) + \frac{9}{25} y^2.
\end{equation*}
In $\Sigma^{-}$, we consider the cubic isochronous center of type \eqref{syst:S3}
\begin{equation}\label{sist2-HL-S3}
    \begin{aligned}
    \dot{x} =& 0.000316235 x^3+x^2 (-0.181944 y-0.358614)+x ((0.500161 y+6.23234) y+10.2014) \\
    & + y (y (3.70271 y+15.2649)+20.176)+8.20727, \\[4pt]
    \dot{y} =& 0.00013467 x^3+x^2 (0.00517374 y+0.00507782)+x ((-0.487029 y-1.98247) y - 1.98832) \\
    & + y (y (-1.67659 y-3.81676)+4.49615)+11.0264,    
   \end{aligned}
\end{equation}
with first integral    
\begin{equation*}
\begin{aligned}
H_3(x,y) &= \frac{1}{\left( - 1 + 3 \left(\frac{11}{245} x + 0.149717 y - \frac{25}{101}\right)^2 \right)^3} \Big( \left(\frac{11}{245} x + 0.149717 y - \frac{25}{101}\right)^2 + \\
& \left(\frac{1}{176} x - 1.2606 y - 2.47408\right)^2  - 4 \left(\frac{11}{245} x + 0.149717 y - \frac{25}{101}\right)^4 \\
& + 4 \left(\frac{11}{245} x + 0.149717 y - \frac{25}{101}\right)^6 \Big).
\end{aligned}
\end{equation*}

Solving system~\eqref{sist_Hl-H3} for $y_1 < y_2$, we obtain three pairs of real solutions $(p_i,q_i)$, where $p_i=(0,x_i)$ and $q_i=(0,y_i)$ for $i=1,\dots,3$, with $x_i<y_i$, given by

\begin{align*}
p_1 =& (0, -1.53678 \ldots), \quad q_1=(0, -0.0187778 \ldots), \\
p_2 =& (0, -1.40003 \ldots), \quad q_2=(0, -0.15553 \ldots), \\
p_3 =& (0, -1.17704 \ldots), \quad q_3=(0, -0.378511 \ldots),
\end{align*}
that provide the crossing limit cycles of discontinuous piecewise differential system \eqref{sist1-HL-S3}–\eqref{sist2-HL-S3} shown in Figure \ref{fig-HL-S3}. 

\begin{figure}
\centering
\includegraphics[scale=0.30]{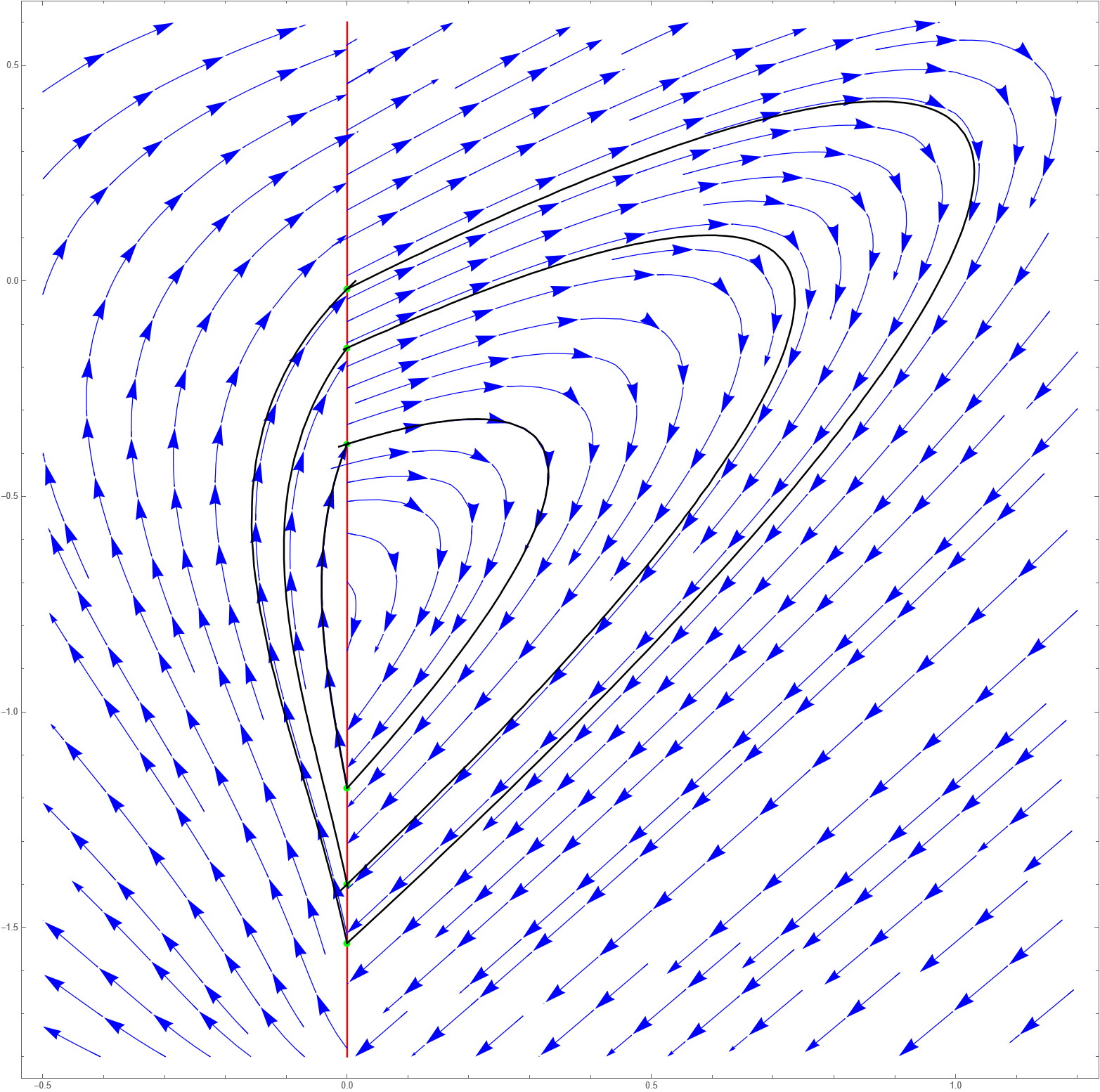}
\caption{The three limit cycle of the discontinuous piecewise differential system \eqref{sist1-HL-S3}-\eqref{sist2-HL-S3} of Theorem \ref{Thm1}.}
\label{fig-HL-S3}
\end{figure}

\subsection*{\textbf{Proof of Theorem \ref{Thm1} for systems  $\mathtt{L}_c- \mathtt{S}_4$}}
We consider in $\Sigma^-$ a cubic system \eqref{syst:S4}, that is, the cubic system $\mathtt{S}_4$ after an arbitrary affine change of variables,  with its first
integral $H_4(x, y)$; and in $\Sigma^+$,  a linear differential system $\mathtt{L}_c$ with its first integral $H_L(x, y)$.

If the discontinuous piecewise differential system $\mathtt{L}_c$-$\eqref{syst:S4}$  admits a limit cycle that intersects the discontinuity straight line  $x=0$ at two points, $(x,0)$ and $(0, y)$, then these two points must satisfy the following system of equations
\begin{align}\label{sist--Hs-H4}
    H_L(0,y_1) - H_L(0,y_2) = 0, \text{ and } 
    H_4(0,y_1) - H_4(0,y_2)  = 0,
\end{align}
or equivalently
\begin{equation*}
    \begin{aligned}
    \mathscr{E}_1(y_1,y_2)=& - (y_1 - y_2)\left(8 B D^2 - 4 A^2 (y_1 + y_2) - \Omega^2 (y_1 + y_2)\right)=0,\\
     \mathscr{E}_2(y_1,y_2)=& \frac{ (y_1 - y_2)\mathscr{P}_4(y_1,y_2) }{\left(1 + 3 (c_1 + b_1 y_1)^2\right)^3 \left(1 + 3 (c_1 + b_1 y_2)^2\right)^3}=0,
    \end{aligned}
\end{equation*}

where  $\mathscr{P}_4(y_1,y_2)$ is a polynomial of degree $7$, given by

\begin{equation*}
\begin{aligned}
\mathscr{P}_4(y_1,y_2)=& 9 b_1^8 y_1^2 y_2^2 (y_1+y_2) (y_1^2+y_2^2) + 18 b_1^7 c_1 y_1 y_2 (y_1^4+4 y_1^3 y_2+4 y_1^2 y_2^2+4 y_1 y_2^3+y_2^4) \\
& + b_1^6 (y_1^5 (9 c_1^2-27 (\gamma_1+\beta_1 y_2)^2+4)+y_1^4 y_2 (117 c_1^2-27 (\gamma_1+\beta_1 y_2)^2+4) \\
& + y_1^3 y_2^2 (252 c_1^2-27 (\gamma_1+\beta_1 y_2)^2+13)+y_1^2 y_2^3 (252 c_1^2-27 (\gamma_1+\beta_1 y_2)^2+13) \\
& +y_1 y_2^4 (117 c_1^2-27 \gamma_1 (\gamma_1+2 \beta_1 y_2)+4)+y_2^5 (9 c_1^2-27 \gamma_1^2+4)) + 6 b_1^5 c_1 (y_1^4 (9 \\
& c_1^2-27 (\gamma_1+\beta_1 y_2)^2+4)+y_1^3 y_2 (54 c_1^2-27 (\gamma_1+\beta_1 y_2)^2+7)+y_1^2 y_2^2 (84 c_1^2-27 \\
& (\gamma_1+\beta_1 y_2)^2+13)+y_1 y_2^3 (54 c_1^2-27 \gamma_1 (\gamma_1+2 \beta_1 y_2)+7)+y_2^4 (9 c_1^2-27 \gamma_1^2+4)) \\
& + b_1^4 (-54 \beta_1 (15 c_1^2+1) \gamma_1 y_1 y_2 (y_1^2+y_1 y_2+y_2^2) - 27 (15 c_1^2+1) \gamma_1^2 (y_1+y_2) \\
& (y_1^2+y_2^2) + (y_1+y_2) ((15 c_1^2+1) y_1^2 (9 c_1^2-27 \beta_1^2 y_2^2+4) + 72 (5 c_1^2+1) c_1^2 y_1 y_2 \\
& + (135 c_1^4+69 c_1^2+4) y_2^2)) + 4 b_1^3 c_1 ((5 c_1^2+1) y_1^2 (9 c_1^2-27 (\gamma_1+\beta_1 y_2)^2+4) \\
& + (5 c_1^2+1) y_2^2 (9 c_1^2-27 \gamma_1^2+4) + y_1 y_2 (99 c_1^4 + c_1^2 (47-135 \gamma_1 (\gamma_1+2 \beta_1 y_2)) \\
& - 27 \gamma_1 (\gamma_1+2 \beta_1 y_2) + 4)) + b_1^2 (-18 \beta_1 (45 c_1^4+18 c_1^2+1) \gamma_1 y_1 y_2 - 9 (45 c_1^4 \\
& +18 c_1^2+1) \gamma_1^2 (y_1+y_2) + (3 c_1^2+1)(42 c_1^4+21 c_1^2+1)(y_1+y_2)) + 2 b_1 c_1 \\
& (3 c_1^2+1)^2 (2 c_1^2-9 \gamma_1^2+9 \beta_1^2 y_1 y_2+1) + \beta_1 (3 c_1^2+1)^3 (2 \gamma_1+\beta_1 (y_1+y_2)).
\end{aligned}
\end{equation*}

Since 
$\left(1 + 3 (c_1 + b_1 y_1)^2\right)^3 \left(1 + 3 (c_1 + b_1 y_2)^2\right)^3 \neq 0$ and $y_1<y_2$, we reduce to equivalent systems 
\begin{equation}\label{my0}
    \begin{aligned}
    \widetilde{\mathscr{E}}(y_1,y_2)= & 8 B D^2 - 4 A^2 (y_1 + y_2) - \Omega^2 (y_1 + y_2)=0,\\
    \widetilde{\mathscr{E}}_2(y_1,y_2)=&\mathscr{P}_4(y_1,y_2)=0.
    \end{aligned}
\end{equation}

As in the previous proof, we assume that this system has finitely many solutions.  From $\widetilde{\mathscr{E}}(y_1,y_2)=0$ we obtain that
\begin{equation}\label{my1}
    \begin{aligned}
        y_1=\frac{8BD^2-4 A^2 y_2 - y_2 \omega^2}{4 A^2+\omega^2}.
    \end{aligned}
\end{equation}

So, if we substitute \eqref{my1} into $\widetilde{\mathscr{E}}_2(y_1,y_2)=0$, we obtain a polynomial $\widetilde{\mathscr{P}}(y_2)$ of degree $6$ in the variable $y_2$, and hence at most six real roots. Therefore, the system has at most six real solutions, and consequently, the discontinuous differential system can have at most three limit cycles.
Now we shall prove that the discontinuous piecewise differential system $\mathtt{L}_c-\eqref{syst:S4}$ separated by the straight line $\Sigma:x=0$, having three limit cycles. 
In $\Sigma^+$, we consider the linear differential center 
\begin{align}\label{sist1-HL-S4}
    \dot{x} = \frac{4}{5}x - \frac{89}{90}y + \frac{3}{5} , \quad \dot{y} = 
 \frac{9}{10}x - \frac{4}{5}y+1,
\end{align} 
with the first integral
\begin{equation*}
H_L(x,y) = 4 \left( \frac{9}{10} x - \frac{4}{5} y \right)^2 + \frac{36}{5} \left( x - \frac{3}{5} y \right) + y^2.
\end{equation*}
In $\Sigma^{-}$, we consider the cubic isochronous center of type \eqref{syst:S4}
\begin{equation}\label{sist2-HL-S4}
    \begin{aligned}
    \dot{x} =& 1.97178 x^3 + x^2 (12.1323 y-5.91966) + x (y (23.8596 y-26.2261)+11.3837) + \\
    & y (y (15.065 y-27.4285)+24.7067) - 8.25811, \\[4pt]
    \dot{y} =& x^2 (3.35082\, -6.56751 y) + x (y (14.5127\, -12.7494 y)-6.44068) + y (y (14.9345\, \\
    & -7.95242 y)-13.595)-1.08392 x^3+4.77085,
\end{aligned}
\end{equation}
with first integral    
\begin{equation*}
\begin{aligned}
H_4(x,y) &= \frac{1}{\left( 1 + 3 \left( \frac{5}{12} x + 0.789303 y - \frac{91}{190} \right)^2 \right)^3} \Bigg( \left(\frac{3}{46} x - 0.00012165 y - 0.192376 \right)^2 + \\
&\quad \left(\frac{5}{12} x + 0.789303 y -\frac{91}{190} \right)^2  + 4 \left(\frac{5}{12} x + 0.789303 y -\frac{91}{190}\right)^4 + \\
&\quad 4 \left(\frac{5}{12} x + 0.789303 y - \frac{91}{190}\right)^6 \Bigg).
\end{aligned}
\end{equation*}

Solving system~\eqref{sist--Hs-H4} for $y_1 < y_2$, we obtain three pairs of real solutions $(p_i,q_i)$, where $p_i=(0,x_i)$ and $q_i=(0,y_i)$ for $i=1,\dots,3$, with $x_i<y_i$, given by

\begin{align*}
p_1 =& (0, 5.52401 \ldots), \quad q_1=(0, 6.7375 \ldots), \\
p_2 =& (0, -1.51668 \ldots), \quad q_2=(0, 2.73016 \ldots), \\
p_3 =& (0, -0.632352 \ldots), \quad q_3=(0, 1.84584 \ldots),
\end{align*}
that provide the crossing limit cycles of discontinuous piecewise differential system \eqref{sist1-HL-S4}–\eqref{sist2-HL-S4} shown in Figure \ref{fig-HL-S4}. 

\begin{figure}
\centering
\includegraphics[scale=0.30]{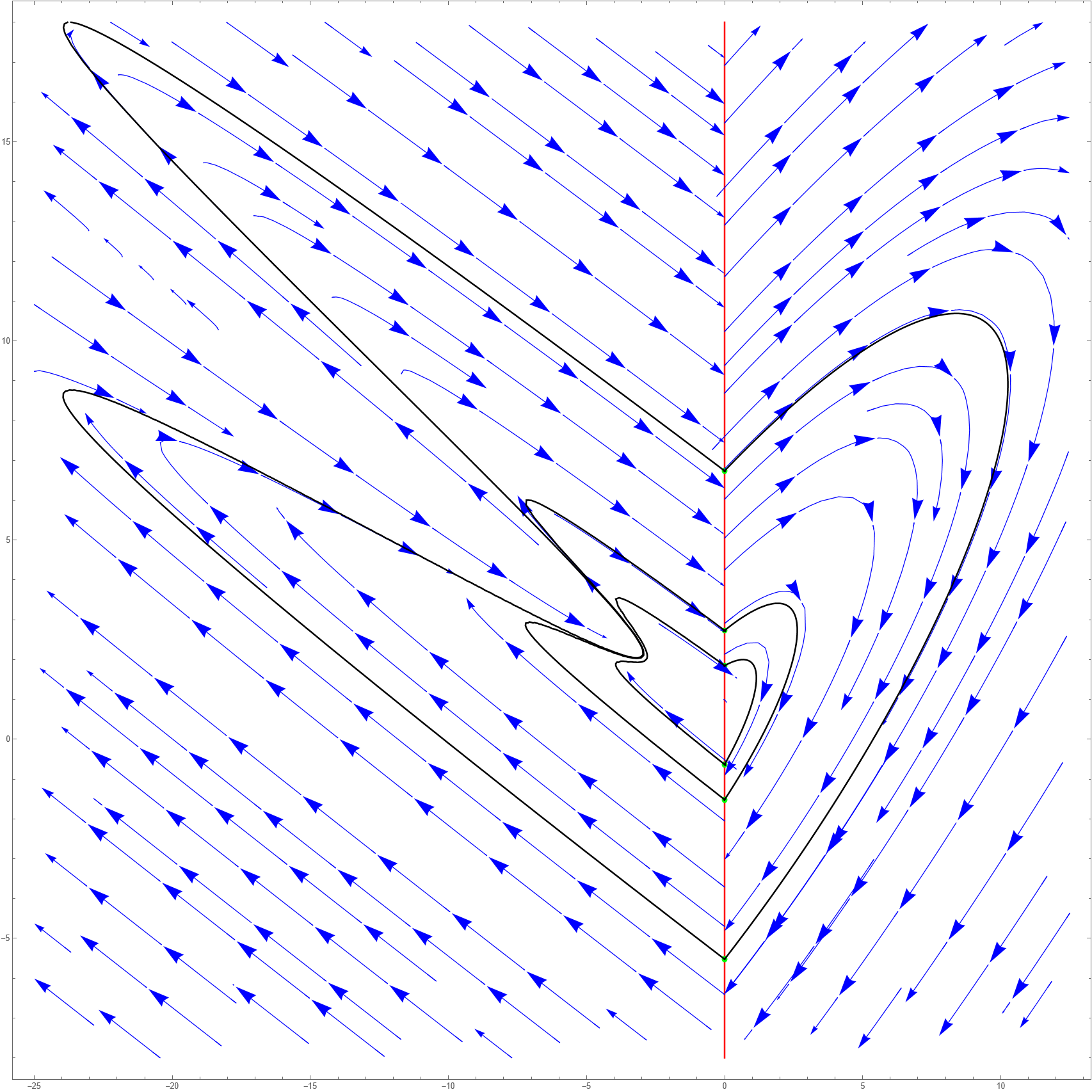}
\caption{The three limit cycle of the discontinuous piecewise differential system \eqref{sist1-HL-S4}-\eqref{sist2-HL-S4} of Theorem \ref{Thm1}.}
\label{fig-HL-S4}
\end{figure}

\subsection*{\textbf{Proof of Theorem \ref{Thm1} for systems  $\mathtt{S}_1- \mathtt{S}_3$}}
We consider in $\Sigma^-$ a cubic system \eqref{syst:S1}, that is, the cubic system $\mathtt{S}_1$ after an arbitrary affine change of variables,  with its first
integral 
$$H_1(x, y)=
\frac{(a x + b y + c)^2+(\alpha x + \beta y + \gamma)^2}
{1+2\,(a x + b y + c)(\alpha x + \beta y + \gamma)},$$ and in $\Sigma^+$,  a cubic system \eqref{syst:S3}, that is, the cubic system $\mathtt{S}_3$ after an arbitrary affine change of variables,  with its first
integral 
$$H_3(x, y)=
\frac{(a_1 x + b_1 y + c_1)^2+(\alpha_1 x + \beta_1 y + \gamma_1)^2
-4(a_1 x + b_1 y + c_1)^4+4(a_1 x + b_1 y + c_1)^6}
{\big(-1+3(a_1 x + b_1 y + c_1)^2\big)^3}.$$

If such discontinuous piecewise differential systems have a limit cycle intersecting the discontinuity straight line $x = 0$ in the two points $(0, y_1 )$ and $(0, y_2)$, then $y_1$ and $y_2$ must satisfy that
\begin{equation}\label{sist_H1-H3}
\begin{aligned}
    H_1(0,y_1) - H_1(0,y_2) &= 0, \text{ and }
    H_3(0,y_1) - H_3(0,y_2) &= 0,
\end{aligned}
\end{equation}
or equivalently
\begin{align*}
\mathscr{E}_2 =& -\frac{(y_1 - y_2) \mathscr{R}_1(y_1,y_2)}{\left(2 b \beta  y_1^2+2 b \gamma  y_1+2 c \gamma +2 \beta  c y_1+1\right) \left(2 b \beta  y_2^2+2 b \gamma  y_2+2 c \gamma +2 \beta  c y_2+1\right)} = 0, \\
\mathscr{E}_3 =& -\frac{(y_1 - y_2) \mathscr{R}_3(y_1,y_2)}{\left(3 b_1^2 y_1^2+6 b_1 c_1 y_1+3 c_1^2-1\right)^3 \left(3 b_1^2 y_2^2+6 b_1 c_1 y_2+3 c_1^2-1\right)^3} = 0,
\end{align*}
where
\begin{align*}
\mathscr{R}_1(y_1,y_2) =& -2 b^3 \gamma  y_1 y_2-2 b^2 c \gamma  y_1+2 \beta  b^2 c y_1 y_2-2 b^2 c \gamma  y_2-b^2 y_1-b^2 y_2+2 b \gamma ^3 \\
& -2 b c^2 \gamma +2 \beta  b c^2 y_1+2 \beta  b c^2 y_2-2 b c-2 \beta  \gamma +2 \beta  b \gamma ^2 y_1 +2 \beta ^2 b \gamma  y_1 y_2 \\
& +2 \beta  b \gamma ^2 y_2+2 \beta  c^3-2 \beta  c \gamma ^2-2 \beta ^2 c \gamma  y_1-2 \beta ^3 c y_1 y_2-2 \beta ^2 c \gamma  y_2-\beta ^2  y_1 \\
&-\beta ^2 y_2 ,
\end{align*}

\begin{align*}
\mathscr{R}_3(y_1,y_2) =& 9 b_1^8 y_1^2 y_2^2 (y_1+y_2) (y_1^2+y_2^2) + 18 b_1^7 c_1 y_1 y_2 (y_1^4+4 y_1^3 y_2+4 y_1^2 y_2^2+4 y_1 y_2^3+y_2^4) \\
& + b_1^6 ((y_1+y_2) (y_1^4 (9 c_1^2-27 \beta_1^2 y_2^2-4) + 108 c_1^2 y_1^3 y_2 + y_1^2 y_2^2 (144 c_1^2-27 \beta_1^2 y_2^2 \\
& -13) + 108 c_1^2 y_1 y_2^3 + (9 c_1^2-4) y_2^4) - 27 \gamma_1^2 (y_1+y_2) (y_1^2-y_1 y_2+y_2^2) (y_1^2+ \\
& y_1 y_2+y_2^2) - 54 \beta_1 \gamma_1 y_1 y_2 (y_1^4+y_1^3 y_2+y_1^2 y_2^2+y_1 y_2^3+y_2^4)) + 6 b_1^5 c_1 (y_1^4 (9 c_1^2 \\
& -27 (\gamma_1+\beta_1 y_2)^2-4) - y_1^3 y_2 (-54 c_1^2+27 (\gamma_1+\beta_1 y_2)^2+7) - y_1^2 y_2^2 (-84 \\
& c_1^2+27 (\gamma_1+\beta_1 y_2)^2+13) - y_1 y_2^3 (-54 c_1^2+27 \gamma_1 (\gamma_1+2 \beta_1 y_2)+7) + y_2^4 (9 \\
& c_1^2-27 \gamma_1^2-4)) + b_1^4 (-54 \beta_1 (15 c_1^2-1) \gamma_1 y_1 y_2 (y_1^2+y_1 y_2+y_2^2) - 27 (15 c_1^2 \\
& -1) \gamma_1^2 (y_1+y_2) (y_1^2+y_2^2) + (y_1+y_2) ((15 c_1^2-1) y_1^2 (9 c_1^2-27 \beta_1^2 y_2^2-4) + \\
& 72 (5 c_1^2-1) c_1^2 y_1 y_2 + (135 c_1^4-69 c_1^2+4) y_2^2)) + 4 b_1^3 c_1 ((5 c_1^2-1) y_1^2 (9 c_1^2-27 \\
& (\gamma_1+\beta_1 y_2)^2-4) + (5 c_1^2-1) y_2^2 (9 c_1^2-27 \gamma_1^2-4) + y_1 y_2 (99 c_1^4 - c_1^2 (135 \gamma_1 \\
& (\gamma_1+2 \beta_1 y_2)+47) + 27 \gamma_1 (\gamma_1+2 \beta_1 y_2) + 4)) + b_1^2 (3 c_1^2-1) (18 \beta_1 (1-15 c_1^2) \\
& \gamma_1 y_1 y_2 - 9 (15 c_1^2-1) \gamma_1^2 (y_1+y_2) + (42 c_1^4-21 c_1^2+1) (y_1+y_2)) + 2 b_1 c_1 \\
& (1-3 c_1^2)^2 (2 c_1^2-9 \gamma_1^2+9 \beta_1^2 y_1 y_2-1) + \beta_1 (3 c_1^2-1)^3 (2 \gamma_1+\beta_1 (y_1+y_2)),
\end{align*}
where $\mathscr{R}_1(y_1,y_2)$ and $\mathscr{R}_3(y_1,y_2)$ are polynomial of degree $2$ and $7$ respectability. Since 
$$\left(2 b \beta y_1^2 + 2 b \gamma y_1 + 2 c \gamma + 2 \beta c y_1 + 1\right)\left(2 b \beta y_2^2 + 2 b \gamma y_2 + 2 c \gamma + 2 \beta c y_2 + 1\right) \neq 0,$$
$$\left(3 (b_1 y_1 + c_1)^2 - 1\right)^3 \left(3 (b_1 y_2 + c_1)^2 - 1\right)^3 \neq 0,$$ 
and $y_1 < y_2$, we can remove these terms to solve the system, and we get the equivalent system
\begin{equation}\label{simplif_sist_h1-h3}
\begin{aligned}
\mathscr{\tilde{E}}_2 (y_1,y_2) =& \mathscr{R}_1(y_1,y_2) = 0, \\
\mathscr{\tilde{E}}_3 (y_1,y_2) =& \mathscr{R}_3(y_1,y_2) = 0.
\end{aligned}
\end{equation}
Eventually, system \eqref{sist_H1-H3} could have a continuum of solutions $(y_1,y_2)$, but then the possible periodic solutions would not be limit cycles. Therefore, we assume that this system has finitely many solutions. From $\mathscr{\tilde{E}}_1 (y_1,y_2)$ we obtain that
\begin{equation}\label{Sol_y1}
\begin{aligned}
y_1 =& - \frac{\mathscr{P}_{13}(y_2)}{\mathscr{Q}_{13}(y_2)},
\end{aligned}
\end{equation}
where
\begin{equation*}
\begin{aligned}
\mathscr{P}_{13}(y_2) =& b^2 \big(y_2 + 2 c y_2 \gamma\big) + \beta \big(-2 c^3 + y_2 \beta + 2 \gamma + 2 c \gamma (y_2 \beta + \gamma)\big) - 2 b \big(-c + c^2 (y_2 \beta - \gamma) \\
& + \gamma^2 (y_2 \beta + \gamma)\big), \\
\mathscr{Q}_{13}(y_2) =& 2 b^3 y_2 \gamma + b^2 \big(1 + 2 c (-y_2 \beta + \gamma)\big) + \beta^2 \big(1 + 2 c (y_2 \beta + \gamma)\big) - 2 b \beta \big(c^2 + \gamma (y_2 \beta + \gamma)\big),
\end{aligned}
\end{equation*}
with $\mathscr{Q}_{13}(y_2)\neq 0$. So if we substitute \eqref{Sol_y1} in $\mathscr{\tilde{E}}_2 (y_1,y_2) = 0$, then we get a polynomial $\mathscr{P}_3(y_2)$ of degree $10$ in the variable $y_2$, and $\mathscr{P}_3(y_2)$ has at most ten roots. Therefore the system \eqref{simplif_sist_h1-h3} has at most five solutions, and consequently, the discontinuous piecewise differential system can have at most five limit cycles.

Now we shall prove that the discontinuous piecewise differential system $\eqref{syst:S1}-\eqref{syst:S3}$ separated by the straight line $\Sigma:x=0$, having three limit cycles. 
In $\Sigma^+$, we consider the cubic isochronous center of type \eqref{syst:S1}
\begin{equation}\label{sist1-S1-S3}
    \begin{aligned}
    \dot{x} =& \frac{1}{900}\left(48(9x-14)y^2 - 144xy - 3x(27x(3x-14)+184) + 8(278y-163)\right), \\[4pt]
    \dot{y} =& \frac{1}{600}\left(27x^2(1-6y) + 6x(84y-139) + 16y(9y(2y-1)-73) + 1364\right),
\end{aligned}
\end{equation}
with the first integral
\begin{equation*}
\begin{aligned}
  H_1(x,y) = \frac{\left(\frac{3}{5}  + \frac{2}{5} y - 1\right)^2 + \left(\frac{3}{10} x + \frac{4}{5} y - \frac{3}{5}\right)^2}{1 + 2\left(\frac{3}{5} x + \frac{2}{5} y - 1\right)\left(\frac{3}{10} x + \frac{4}{5} y - \frac{3}{5}\right)}.
\end{aligned}
\end{equation*}
In $\Sigma^{-}$, we consider the cubic isochronous center of type \eqref{syst:S3}
\begin{equation}\label{sist2-S1-S3}
    \begin{aligned}
    \dot{x} =& \frac{1}{1100} \big( 756x^3 - 6x^2(67y+726) + x(1837 - 12y(411y+781)) \\
    & + y(18y(110-111y) + 8129) + 1015 \big), \\[4pt]
    \dot{y} =& \frac{1}{1100} \big( 142x^3 + 6x^2(62y-55) - x(114y(15y+44) + 1441) \\
    & + y(803 - 342y(6y+11)) + 705 \big),  
   \end{aligned}
\end{equation}
with first integral    
\begin{equation*}
    \begin{aligned}
H_3(x,y) =& \frac{ \left(-\frac{1}{5} x + \frac{3}{5} y + 1\right)^2 + \left(-\frac{1}{2} x - \frac{7}{10} y + \frac{1}{10}\right)^2 - 4\left(-\frac{1}{5} x + \frac{3}{5} y + 1\right)^4 + 4\left(-\frac{1}{5} x + \frac{3}{5} y + 1\right)^6 }{\left(- 1 + 3\left(-\frac{1}{5} x + \frac{3}{5} y + 1\right)^2 \right)^3}.
   \end{aligned}
\end{equation*}

Solving system~\eqref{sist_H1-H3} for $y_1 < y_2$, we obtain three pairs of real solutions $(p_i,q_i)$, where $p_i=(0,x_i)$ and $q_i=(0,y_i)$ for $i=1,\dots,3$, with $x_i<y_i$, given by
\begin{align*}
p_1 =& (0, -0.272336 \ldots), \quad q_1=(0, 1.24079 \ldots), \\
p_2 =& (0, 1.75938 \ldots), \quad q_2=(0, 9.2799 \ldots), \\
p_3 =& (0, 1.92051 \ldots), \quad q_3=(0, 4.65668 \ldots),
\end{align*}
that provide the crossing limit cycles of discontinuous piecewise differential system \eqref{sist1-S1-S3}–\eqref{sist2-S1-S3} shown in Figure \ref{fig-S1-S3}. 

\begin{figure}
\centering
\includegraphics[scale=0.30]{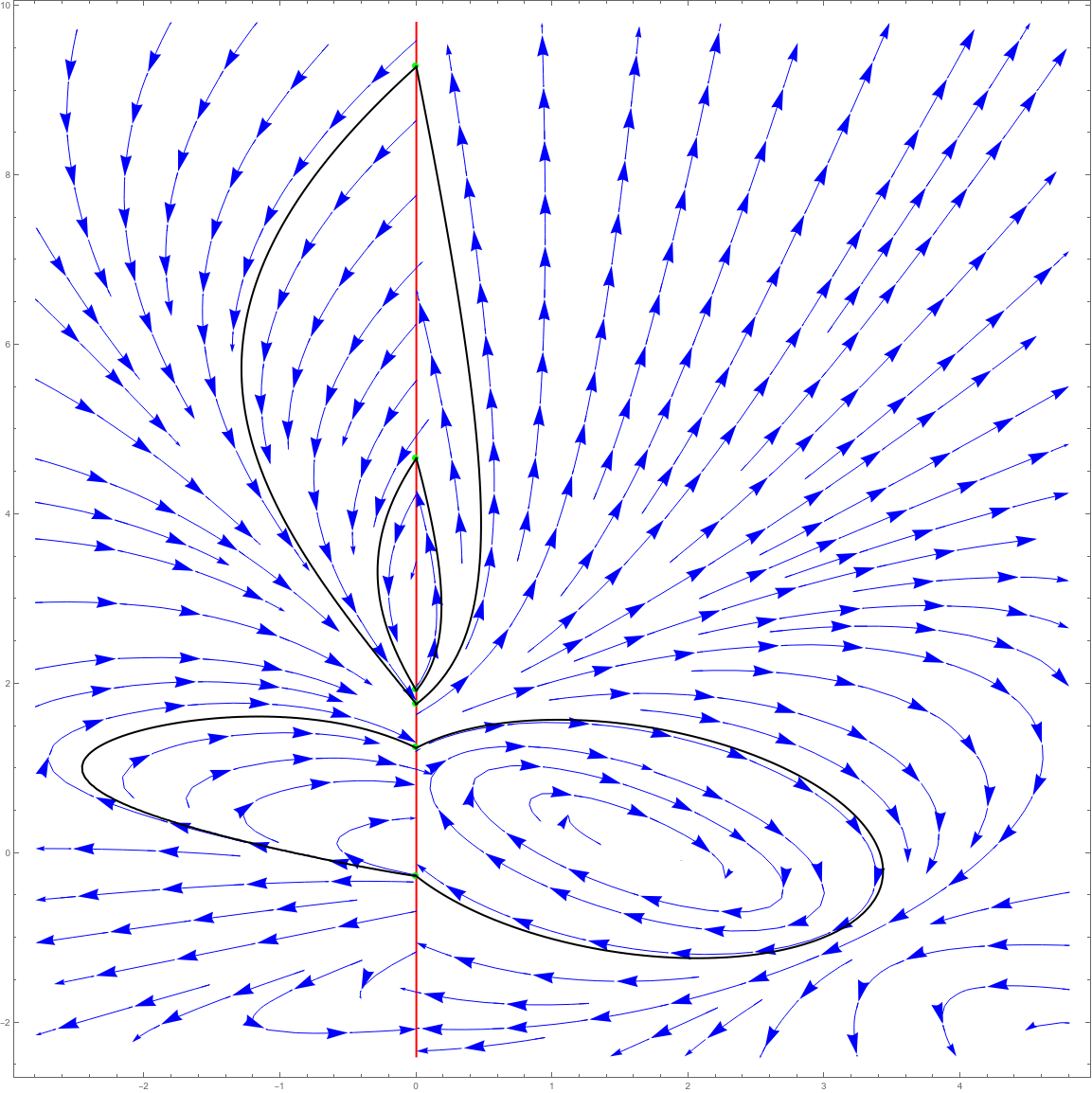}
\caption{The three limit cycle of the discontinuous piecewise differential system \eqref{sist1-S1-S3}-\eqref{sist2-S1-S3} of Theorem \ref{Thm1}.}
\label{fig-S1-S3} 
\end{figure}

\subsection*{\textbf{Proof of Theorem \ref{Thm1} for systems  $\mathtt{S}_1- \mathtt{S}_4$.}}
We consider in $\Sigma^-$ a cubic system \eqref{syst:S1}, that is, the cubic system $\mathtt{S}_1$ after an arbitrary affine change of variables,  with its first
integral $$H_1(x, y)=
\frac{(a x + b y + c)^2+(\alpha x + \beta y + \gamma)^2}
{1+2\,(a x + b y + c)(\alpha x + \beta y + \gamma)},$$ 
and in $\Sigma^+$,  a cubic system \eqref{syst:S4}, that is, the cubic system $\mathtt{S}_4$ after an arbitrary affine change of variables,  with its first
integral $$H_4(x, y)=
\frac{(a_1 x + b_1 y + c_1)^2+(\alpha_1 x + \beta_1 y + \gamma_1)^2
+4(a_1 x + b_1 y + c_1)^4+4(a_1 x + b_1 y + c_1)^6}
{\big(1+3(a_1 x + b_1 y + c_1)^2\big)^3}.$$

If the discontinuous piecewise differential system $H_1$-$\eqref{syst:S4}$  admits a limit cycle that intersects the discontinuity straight line  $x=0$ at two points, $(x,0)$ and $(0, y)$, then these two points must satisfy the following system of equations
\begin{align}\label{sist--H1-H4}
    H_1(0,y_1) - H_1(0,y_2) = 0, \text{ and } 
    H_4(0,y_1) - H_4(0,y_2)  = 0,
\end{align}

or equivalently
\begin{equation*}
    \begin{aligned}
    \mathscr{E}_1(y_1,y_2)=&-\frac{(y_1 - y_2) \mathscr{R}_1(y_1,y_2)}{\left(2 b \beta  y_1^2+2 b \gamma  y_1+2 c \gamma +2 \beta  c y_1+1\right) \left(2 b \beta  y_2^2+2 b \gamma  y_2+2 c \gamma +2 \beta  c y_2+1\right)} =0,\\
     \mathscr{E}_2(y_1,y_2)=& \frac{ (y_1 - y_2)\mathscr{R}_4(y_1,y_2) }{\left(1 + 3c^2 + 6bc\,y_1 + 3b^2 y_1^2\right)^3
\left(1 + 3c^2 + 6bc\,y_2 + 3b^2 y_2^2\right)^3}=0,
    \end{aligned}
\end{equation*}
where  
\begin{equation*}
\begin{aligned}
\mathscr{R}_4(y_1,y_2)=&-2 b^3 y_1 y_2 \gamma- b^2 \left(y_1 + y_2 - 2 c y_1 y_2 \beta + 2 c (y_1 + y_2) \gamma \right) \\
&- \beta \left(-2 c^3 + (y_1 + y_2)\beta + 2 \gamma 
+ 2 c (y_1 \beta + \gamma)(y_2 \beta + \gamma)\right) \\
&+ 2 b \left(-c + c^2 \big((y_1 + y_2)\beta - \gamma\big) 
+ \gamma (y_1 \beta + \gamma)(y_2 \beta + \gamma)\right).
\end{aligned}
\end{equation*}
Here $\mathscr{R}_1(y_1,y_2)$ is a polynomial of degree $2$ and  $\mathscr{R}_4(y_1,y_2)$ is a polynomial of degree $7$. 

The system~\eqref{sist--H1-H4} could have a continuum of solutions $(y_1,y_2)$, but then the possible periodic solutions would not be limit cycles. Therefore, we assume that this system has finitely many solutions.

From $\widetilde{\mathscr{E}}_2(y_1,y_2)=0$ we obtain that
\begin{equation}\label{my11}
\begin{aligned}
y_1=&- \dfrac{\mathscr{P}_1(y_2)}{\mathscr{Q}_1(y_2)},
\end{aligned}
\end{equation}
where 
\begin{equation}
    \begin{aligned}
     \mathscr{P}_1(y_2)=& b^2 \big(y_2 + 2 c y_2 \gamma\big)
+ \beta \big(-2 c^3 + y_2 \beta + 2 \gamma + 2 c \gamma (y_2 \beta + \gamma)\big)
- 2 b \big(-c + c^2 (y_2 \beta - \gamma)\\
& + \gamma^2 (y_2 \beta + \gamma)\big), \\
\mathscr{Q}_1(y_2)=& 2 b^3 y_2 \gamma
+ b^2 \big(1 + 2 c (-y_2 \beta + \gamma)\big)
+ \beta^2 \big(1 + 2 c (y_2 \beta + \gamma)\big)
- 2 b \beta \big(c^2 + \gamma (y_2 \beta + \gamma)\big).
    \end{aligned}
\end{equation}

So, if we substitute \eqref{my11} into $\widetilde{\mathscr{E}}_2(y_1,y_2)=0$, we obtain a polynomial $\widetilde{\mathscr{P}}(y_1,y_2)$ of degree $10$ in the variable $y_2$, and hence at most ten real roots. Therefore, the system has at most ten real solutions, and consequently, the discontinuous differential system can have at most five limit cycles.

Now we shall prove that the discontinuous piecewise differential system $\eqref{syst:S1}-\eqref{syst:S4}$ separated by the straight line $\Sigma:x=0$, having three limit cycles. 
In $\Sigma^+$, we consider the cubic isochronous center of type \eqref{syst:S1}
\begin{equation}\label{sist1-S1-S4}
    \begin{aligned}
    \dot{x} =& \frac{-153 (14 x+9) y^2+36 x (140 x+27) y+12 x (9 x (33-14 x)+743)-16358 y-15300}{4200}, \\[4pt]
    \dot{y} =& \frac{1}{700} \left(-36 x^2 (7 y+9)+12 x (y (70 y+153)+131)-y \left(357 y^2+837 y+1486\right)-900\right),
\end{aligned}
\end{equation}
with the first integral
\begin{equation*}
H_1(x,y) = \frac{\left(\frac{3}{5} x - y - \frac{9}{10}\right)^2 + \left(-\frac{7}{10} y - \frac{9}{10}\right)^2}{1 + 2 \left(-\frac{7}{10} y - \frac{9}{10}\right) \left(\frac{3}{5} x - y - \frac{9}{10}\right)}.
\end{equation*}
In $\Sigma^{-}$, we consider the cubic isochronous center of type \eqref{syst:S4}
\begin{equation}\label{sist2-S1-S4}
    \begin{aligned}
    \dot{x} =& \frac{1914 x^3+3 x^2 (3959 y+2732)+6 x (3 y (109 y+879)+1147)+y (377-54 y (44 y+31))-40}{3500}, \\[4pt]
    \dot{y} =& \frac{-214 x^3-18 x^2 (39 y+37)+9 x y (407 y+232)+3718 x+3 y (9 y (38 y+127)+1066)+2970}{3500},
   \end{aligned}
\end{equation}
with first integral    
\begin{equation*}
    \begin{aligned}
H_4(x,y) =& \frac{\left(-\frac{3}{5} x - \frac{1}{10} y - \frac{2}{5}\right)^2 + \left(\frac{1}{10} x + \frac{3}{5} y + \frac{3}{10}\right)^2 + 4 \left(\frac{1}{10} x + \frac{3}{5} y + \frac{3}{10}\right)^4 4 \left(\frac{1}{10} x + \frac{3}{5} y + \frac{3}{10}\right)^6}{\left( 1 + 3 \left(\frac{1}{10} x + \frac{3}{5} y + \frac{3}{10}\right)^2\right)^3}.
   \end{aligned}
\end{equation*}

Solving system~\eqref{sist--H1-H4} for $y_1 < y_2$, we obtain three pairs of real solutions $(p_i,q_i)$, where $p_i=(0,x_i)$ and $q_i=(0,y_i)$ for $i=1,\dots,3$, with $x_i<y_i$, given by
\begin{align*}
p_1 =& (0, -2.41934 \ldots), \quad q_1=(0, 0.925787 \ldots), \\
p_2 =& (0, -2.07259 \ldots), \quad q_2=(0, 0.310176 \ldots), \\
p_3 =& (0, -1.55539 \ldots), \quad q_3=(0, -0.427096\ldots),
\end{align*}
that provide the crossing limit cycles of discontinuous piecewise differential system \eqref{sist1-S1-S4}–\eqref{sist2-S1-S4} shown in Figure \ref{fig-S1-S4}. 

\begin{figure}
\centering
\includegraphics[scale=0.30]{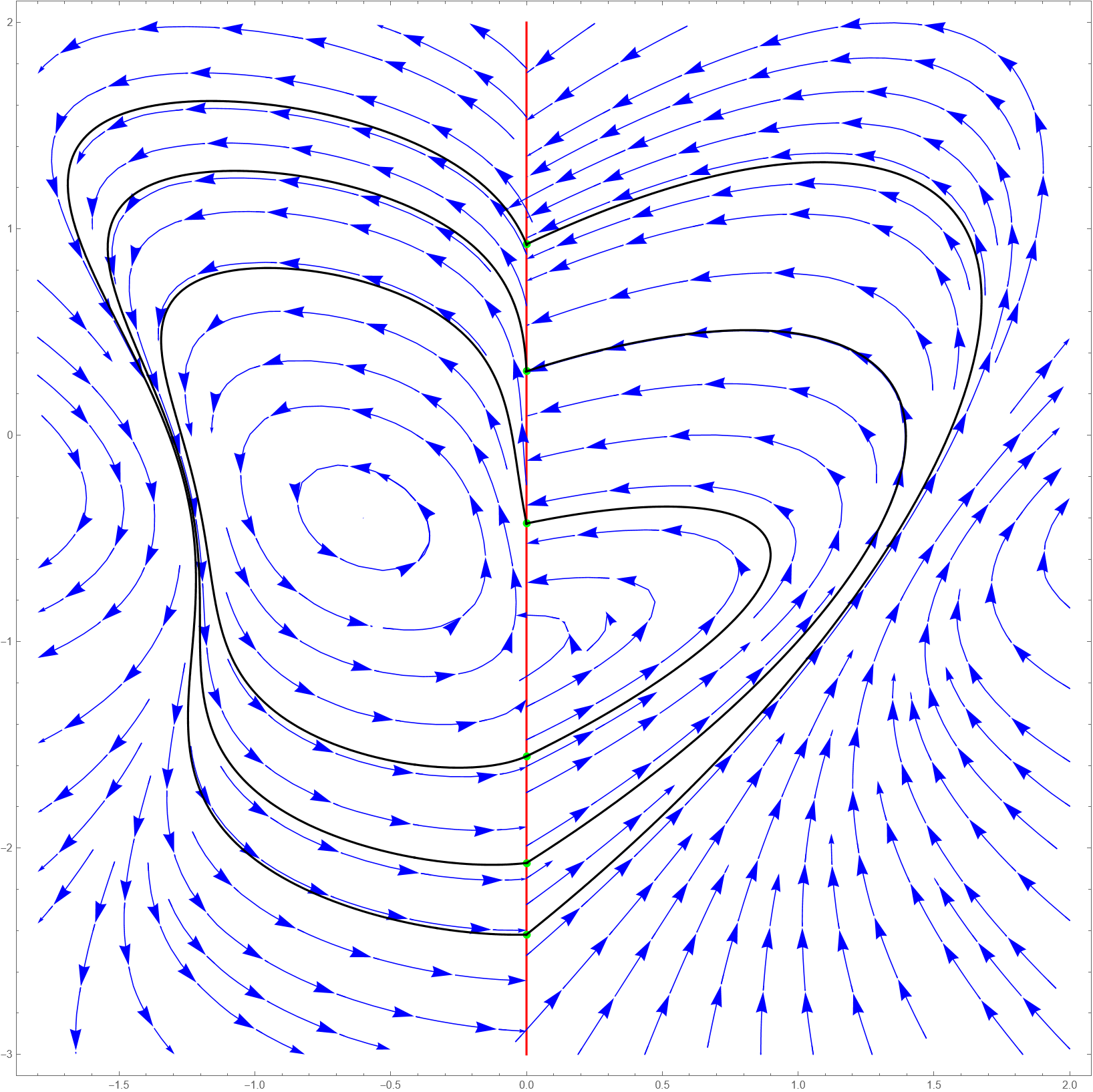}
\caption{The three limit cycle of the discontinuous piecewise differential system \eqref{sist1-S1-S4}-\eqref{sist2-S1-S4} of Theorem \ref{Thm1}.}
\label{fig-S1-S4} 
\end{figure}

\subsection*{\textbf{Proof of Theorem \ref{Thm1} for systems  $\mathtt{S}_2- \mathtt{S}_i$, $i=2,3,4$.}}
We consider in $\Sigma^+$ the  cubic isochronous centers  $(\mathtt{S}_2)$ after an arbitrary affine change of variables,  with its first
integral $H_2(x, y)$; and in $\Sigma^-$,  a  the  cubic isochronous centers $(\mathtt{S}_i)$ with its first integral $H_i(x, y)$, for $i=2,3$ and $4$, where
\begin{equation*}\label{sistemamud}
\begin{aligned}
H_2(x,y)
&=
\frac{(a_1 x + b_1 y + c_1)^2+(\alpha_1 x + \beta_1 y + \gamma_1)^2}
{1+2\,(a_1 x + b_1 y + c_1)(\alpha_1 x + \beta_1 y + \gamma_1)}, \\[0.5em]
\widetilde{H_2}(x,y)
&=
\frac{\big((a_2 x + b_2 y + c_2)^2+(\alpha_2 x + \beta_2 y + \gamma_2)^2\big)^2}
{1+4\,(a_2 x + b_2 y + c_2)(\alpha_2 x + \beta_2 y + \gamma_2)}, \\[0.5em]
\widetilde{H_3}(x,y)
&=
\frac{(a_3 x + b_3 y + c_3)^2+(\alpha_3 x + \beta_3 y + \gamma_3)^2
-4(a_3 x + b_3 y + c_3)^4+4(a_3 x + b_3 y + c_3)^6}
{\big(-1+3(a_3 x + b_3 y + c_3)^2\big)^3}, \\[0.5em]
\widetilde{H_4}(x,y)
&=
\frac{(a_4 x + b_4 y + c_4)^2+(\alpha_4 x + \beta_4 y + \gamma_4)^2
+4(a_4 x + b_4 y + c_4)^4+4(a_4 x + b_4 y + c_4)^6}
{\big(1+3(a_4 x + b_4 y + c_4)^2\big)^3}.
\end{aligned}
\end{equation*}

If the discontinuous piecewise differential system $(\mathtt{S}_2)$-$(\mathtt{S}_i)$ for $i=2,3, 4$ admits a limit cycle that intersects the nonregular line  $\Sigma$ at two points, $(0,y_1)$ and $(0, y_2)$, then these two points must satisfy the following system of equations:
\begin{align}\label{sist_H_Hi}
 H_2(0,y_1) - H_2(0,y_2)=& \frac{(y_2-y_1)\mathscr{M}_2(x,y)}{\left(1 + 4 (c_1 + b_1 y_1)(y_1 \beta_1 + \gamma_1)\right)
\left(1 + 4 (c_1 + b_1 y_2)(y_2 \beta_1 + \gamma_1)\right)} = 0, \nonumber \\
 \widetilde{H_i}(0,y_1) - \widetilde{H_i}(0,y_2) =& \frac{(y_1-y_2)\mathscr{N}_{i}(x,y)}{D_{i}(x,y)} = 0,
\end{align}
where 
 \begin{equation*}
     \begin{aligned}
         D_2(x,y)=& \left(1 + 4 (c_2 + b_2 y_1)(y_1 \beta_2 + \gamma_2)\right)
\left(1 + 4 (c_2 + b_2 y_2)(y_2 \beta_2 + \gamma_2)\right), \\
         D_3(x,y)=& \left(-1 + 3 (c_3 + b_3 y_1)^2\right)^3
\left(-1 + 3 (c_3 + b_3 y_2)^2\right)^3, \\
         D_4(x,y)=& \left(1 + 3 (c_4 + b_4 y_1)^2\right)^3
\left(1 + 3 (c_4 + b_4 y_2)^2\right)^3.
     \end{aligned}
 \end{equation*}
Moreover, $\mathscr{M}_2(x,y)$ has degree five, and for $i=2,3,4$, the polynomials $\mathscr{N}_i(x,y)$ have degrees $5$, $7$, and $7$, respectively.

 
The number of common zeros $(x, y)$ of the polynomials $\mathscr{M}_2$ and $\mathscr{N}_i$ determines the existence and number of limit cycles in the discontinuous piecewise differential system  $(\mathtt{S}_2)$-$(\mathtt{S}_i)$. We compute the two resultants 
\[
\mathcal{R}^{i}_{x} = \operatorname{Res}(\mathscr{M}_2, \mathscr{N}_i, x) \quad \text{and} \quad \mathcal{R}^{i}_{y} = \operatorname{Res}(\mathscr{M}_2, \mathscr{N}_i, y),
\]
with respect to $x$ and $y$, respectively. Since $\mathscr{M}_2$ and $\mathscr{N}_i$ are symmetric in $x$ and $y$, the resultants $\mathcal{R}^{i}_{x}$ and $\mathcal{R}^{i}_{y}$ coincide. It therefore suffices to compute one of them. We focus on $\mathcal{R}^{i}_{x}$, which is a polynomial in $y$ of degree $16$, $26$, and $26$ for $i=2,3,4$, respectively. Due to the size and complexity of its expression, we omit it here.
Consequently, the maximum number of solutions of system~\eqref{sist_H_Hi} is at most $8$, $13$ and $13$  respectively.

Now we shall prove that the discontinuous piecewise differential system \eqref{syst:S2}-\eqref{syst:S2} separated by the straight line $\Sigma:x=0$, having three limit cycles. 

In $\Sigma^+$, we consider the cubic isochronous center of type \eqref{syst:S2}
\begin{equation}\label{sist1-S2-S2}
    \begin{aligned}
    \dot{x} =& \frac{1}{4800} \big( -11808 x^3+24 x^2 (599-1950 y)+8 x (195 y (5-33 y)+472)- \\
    & 130 y (3 y (42 y+17)-20)+717 \big), \\[4pt]
    \dot{y} =& \frac{1}{2400} \big( 42 (612 x+197) y^2+8 (51 x (66 x+17)-74) y+680 x (3 x (4 x-1)-2) \\
    & +7044 y^3+465 \big),    
\end{aligned}
\end{equation}
with the first integral
\begin{equation*}
H_2(x,y) = \frac{\left(\left(\frac{1}{5} x + \frac{2}{5} y - \frac{3}{10}\right)^2 + \left(\frac{1}{10} x + \frac{3}{10} y - 1\right)^2\right)^2}{1 + 4 \left(\frac{1}{5} x + \frac{2}{5} y - \frac{3}{10}\right) \left(\frac{1}{10} x + \frac{3}{10} y - 1\right)}.
\end{equation*}

In $\Sigma^{-}$, we consider the cubic isochronous center of type \eqref{syst:S2}
\begin{equation}\label{sist2-S2-S2}
    \begin{aligned}
    \dot{x} =& \frac{1}{3500} \big( -1565 x^3+15 x^2 (864 y-893)-15 x (6 y (57 y+2)-541) - \\
    & 30 y (9 y (6 y-31)+436)+5893 \big), \\[4pt]
    \dot{y} =& \frac{1}{10500} \big( -20880 x^3+435 x^2 (57 y+1)+290 x (9 y (9 y-31)+218) - \\
    & 45 y (6 y (37 y-92)+331)-5297 \big),  
   \end{aligned}
\end{equation}
with first integral    
\begin{equation*}
\widetilde{H_2}(x,y) = \frac{\Big( \big(\frac{1}{2} x + \frac{1}{10} y - \frac{1}{2}\big)^2 + \big(-\frac{4}{5} x - \frac{3}{10} y - \frac{1}{10}\big)^2 \Big)^2}{1 + 4 \big(\frac{1}{2} x + \frac{1}{10} y - \frac{1}{2}\big) \big(-\frac{4}{5} x - \frac{3}{10} y - \frac{1}{10}\big)}.
\end{equation*}

Solving system~\eqref{sist_H_Hi} for $y_1 < y_2$, we obtain three pairs of real solutions $(p_i,q_i)$, where $p_i=(0,x_i)$ and $q_i=(0,y_i)$ for $i=1,\dots,3$, with $x_i<y_i$, given by
\begin{align*}
p_1 =& (0, -1.3786 \ldots), \quad q_1=(0, 3.65369 \ldots), \\
p_2 =& (0, -0.98002 \ldots), \quad q_2=(0, 2.38844 \ldots), \\
p_3 =& (0, -0.86839\ldots), \quad q_3=(0, 2.15357\ldots),
\end{align*}
that provide the crossing limit cycles of discontinuous piecewise differential system \eqref{sist1-S2-S2}–\eqref{sist2-S2-S2} shown in Figure \ref{fig-S2-S2}. 

\begin{figure}
\centering
\includegraphics[scale=0.30]{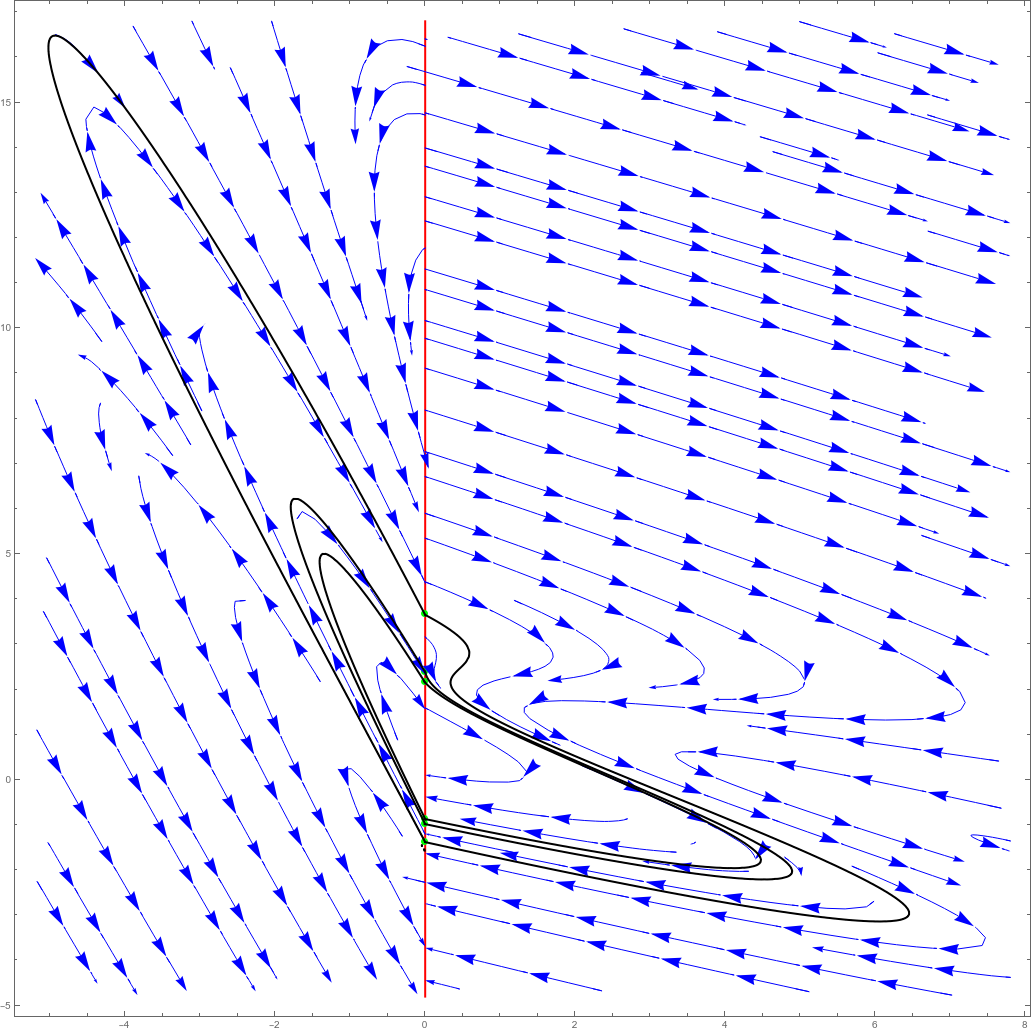}
\caption{The three limit cycle of the discontinuous piecewise differential system \eqref{sist1-S2-S2}-\eqref{sist2-S2-S2} of Theorem \ref{Thm1}.}
\label{fig-S2-S2} 
\end{figure}

Now we shall prove that the discontinuous piecewise differential system $\eqref{syst:S2}-\eqref{syst:S3}$ separated by the straight line $\Sigma:x=0$, having three limit cycles. 

In $\Sigma^+$, we consider the cubic isochronous center of type \eqref{syst:S2}
\begin{equation}\label{sist1-S2-S3}
\begin{aligned}
   \dot{x} =& \frac{1}{2650} \big( 33466 x^3-120 x^2 (65 y+56)+5 x (39 y (18-17 y)+943) + \\
    & 65 y \left(63 y^2+87 y-29\right)-1985 \big), \\[4pt]
    \dot{y} =& \frac{1}{2650} \big( 2320 x^3+174 x^2 (17 y-9)-58 x (3 y (63 y+58)-29) + \\
    & y (y (2801 y+7149)+2917)-963 \big), 
    \end{aligned}
\end{equation}
with the first integral
\begin{equation*}
H_2(x,y) = \frac{\left(\left(-x + \frac{9}{10} y + \frac{7}{10}\right)^2 + \left(-\frac{2}{5} x - \frac{7}{10} y - \frac{1}{10}\right)^2\right)^2}{1 + 4 \left(-x + \frac{9}{10} y + \frac{7}{10}\right) \left(-\frac{2}{5} x - \frac{7}{10} y - \frac{1}{10}\right)}.
\end{equation*}

In $\Sigma^{-}$, we consider the cubic isochronous center of type \eqref{syst:S3}
\begin{equation}\label{sist2-S2-S3}
\begin{aligned}
    \dot{x} =& \frac{1}{4700} \big( 2558 x^3-3 x^2 (6975 y+3631)+8 x \left(339 y^2+777 y+481\right) + \\
    & 4 y (8 y (244 y+351)-317)-2516 \big), \\[4pt]
    \dot{y} =& \frac{1}{4700} \big( 214 x^3-6 x^2 (151 y+83)+x (2173-3 y (2237 y+2138)) + \\
    & 4 y (y (526 y+915)+166)-620 \big), 
      \end{aligned}
\end{equation}
with first integral    
\begin{equation*}
\widetilde{H_3}(x,y) = \frac{ \left(-\frac{1}{10} x + \frac{4}{5} y +\frac{2}{5}\right)^2 + \left(-\frac{3}{5} x + \frac{1}{10} y + \frac{1}{10}\right)^2  - 4 \left(-\frac{1}{10} x + \frac{4}{5} y + \frac{2}{5}\right)^4 + 4 \left(-\frac{1}{10} x + \frac{4}{5} y + \frac{2}{5}\right)^6 }{\left(- 1 + 3 \left(-\frac{1}{10} x + \frac{4}{5} y + \frac{2}{5}\right)^2 \right)^3}.
\end{equation*}

Solving system~\eqref{sist_H_Hi} for $y_1 < y_2$, we obtain three pairs of real solutions $(p_i,q_i)$, where $p_i=(0,x_i)$ and $q_i=(0,y_i)$ for $i=1,\dots,3$, with $x_i<y_i$, given by
\begin{align*}
p_1 =& (0, -1.6744 \ldots), \quad q_1=(0, -1.33802 \ldots), \\
p_2 =& (0, -1.00923 \ldots), \quad q_2=(0, -0.00178922 \ldots), \\
p_3 =& (0, 0.387748 \ldots), \quad q_3=(0, 1.13514 \ldots),
\end{align*}
that provide the crossing limit cycles of discontinuous piecewise differential system \eqref{sist1-S2-S3}–\eqref{sist2-S2-S3} shown in Figure \ref{fig-S2-S3}. 

\begin{figure}
\centering
\includegraphics[scale=0.30]{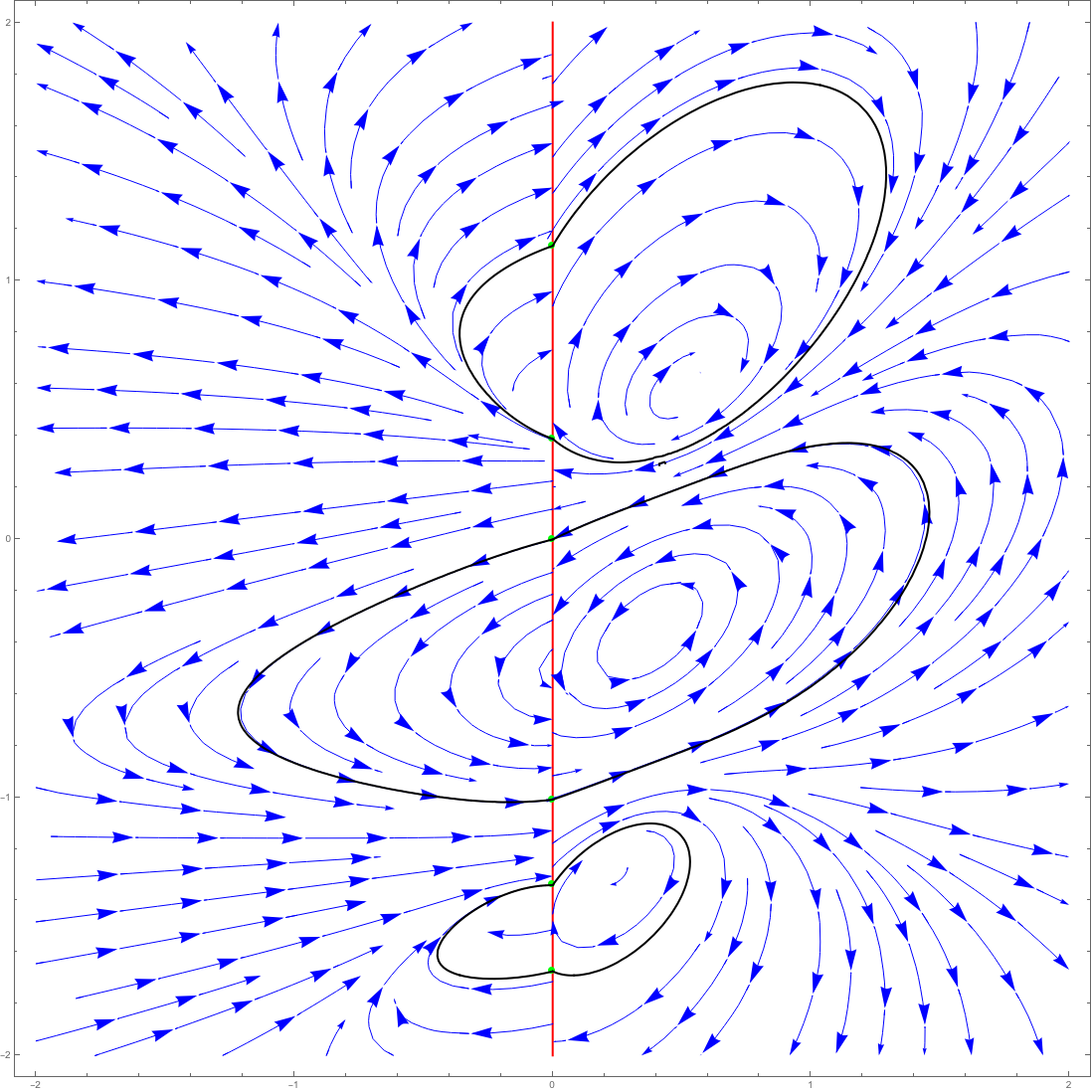}
\caption{The three limit cycle of the discontinuous piecewise differential system \eqref{sist1-S2-S3}-\eqref{sist2-S2-S3} of Theorem \ref{Thm1}.}
\label{fig-S2-S3} 
\end{figure}

Now we shall prove that the discontinuous piecewise differential system $\eqref{syst:S2}-\eqref{syst:S4}$ separated by the straight line $\Sigma:x=0$, having three limit cycles. 

In $\Sigma^+$,  we consider the cubic isochronous center of type \eqref{syst:S2}
\begin{equation}\label{sist1-S2-S4}
    \begin{aligned}
    \dot{x} =& -\frac{6}{25} (2 x+5) (x-2 y+1)^2+\frac{1}{200} (2 x+5)^3+2 x-4 y+2, \\[4pt]
    \dot{y} =& \frac{1}{400} \big( 4 \left(60 x (x+5) y-5 x (x+5) (4 x+1)-64 y^3-144 y^2\right)+268 y+417 \big),    
   \end{aligned}
\end{equation}
with the first integral
\begin{equation*}
H_2(x,y) = \frac{\left(\left(\frac{2}{5} x - \frac{4}{5} y + \frac{2}{5}\right)^2 + \left(-\frac{1}{5} x - \frac{1}{2}\right)^2\right)^2}{1 + 4 \left(-\frac{1}{5} x - \frac{1}{2}\right) \left(\frac{2}{5} x - \frac{4}{5} y + \frac{2}{5}\right)}.
\end{equation*}

In $\Sigma^{-}$, we consider the cubic isochronous center of type \eqref{syst:S4}
\begin{equation}\label{sist2-S2-S4}
    \begin{aligned}
    \dot{x} =& \frac{1}{5300} \big( 6885 x^3+9 x^2 (2806 y+2147)-45 x (y (141 y+854)+694) - \\
    & 25 y (y (218 y+633)+154)+7500 \big), \\[4pt]
    \dot{y} =& \frac{1}{5300} \big( -486 x^3+135 x^2 (271 y+317)+2 x \left(5181 y^2-5331 y-15811\right) - \\
    & 5 (3 y (y (1131 y+4231)+4702)+4084) \big),  
   \end{aligned}
\end{equation}
with first integral    
\begin{equation*}
    \begin{aligned}
\widetilde{H_4}(x,y) =& \frac{\left(-\frac{1}{2} x - \frac{9}{10} y - \frac{4}{5}\right)^2 + \left(\frac{9}{10} x - \frac{1}{2} y - 1\right)^2 + 4 \left(\frac{9}{10} x - \frac{1}{2} y - 1\right)^4 + 4 \left(\frac{9}{10} x - \frac{1}{2} y - 1\right)^6}{\left(1 + 3 \left(\frac{9}{10} x - \frac{1}{2} y - 1\right)^2 \right)^3}.
   \end{aligned}
\end{equation*}

Solving system~\eqref{sist_H_Hi} for $y_1 < y_2$, we obtain three pairs of real solutions $(p_i,q_i)$, where $p_i=(0,x_i)$ and $q_i=(0,y_i)$ for $i=1,\dots,3$, with $x_i<y_i$, given by
\begin{align*}
p_1 =& (0, -1.36147 \ldots), \quad q_1=(0, -0.186783 \ldots), \\
p_2 =& (0, -0.0884511 \ldots), \quad q_2=(0, 3.04379 \ldots), \\
p_3 =& (0, -0.0330296 \ldots), \quad q_3=(0, 2.24123 \ldots),
\end{align*}
that provide the crossing limit cycles of discontinuous piecewise differential system \eqref{sist1-S2-S4}–\eqref{sist2-S2-S4} shown in Figure \ref{fig-S2-S4}. 

\begin{figure}
\centering
\includegraphics[scale=0.35]{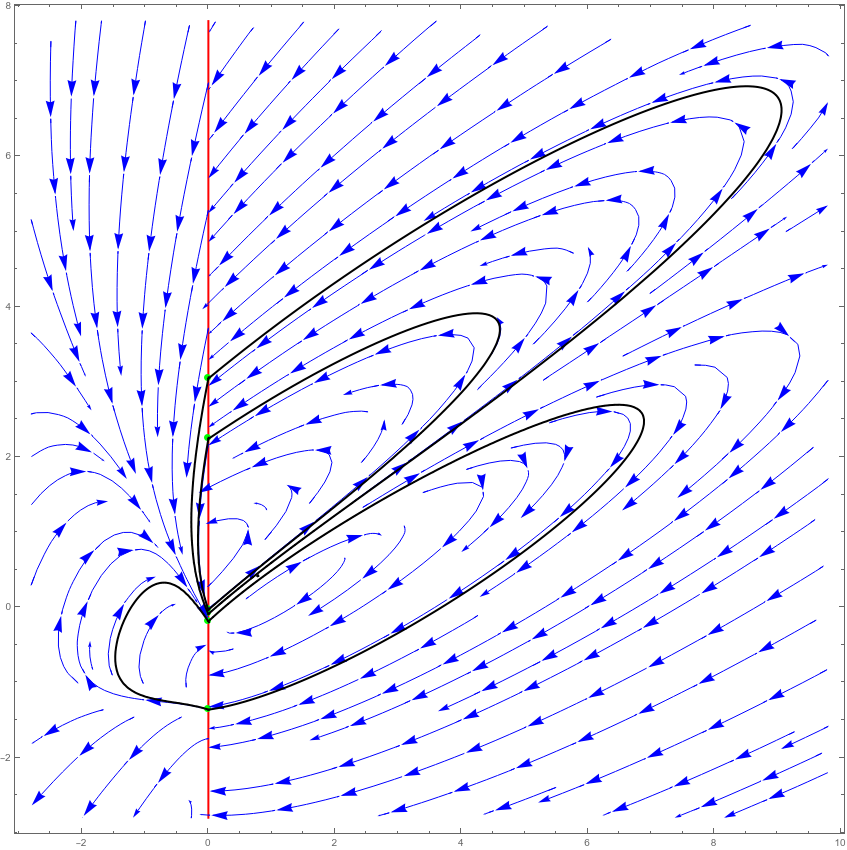}
\caption{The three limit cycle of the discontinuous piecewise differential system \eqref{sist1-S2-S4}-\eqref{sist2-S2-S4} of Theorem \ref{Thm1}.}
\label{fig-S2-S4} 
\end{figure}


\section{Proof of Theorem $\ref{Thm2}$}
\label{sect:04}

\subsection*{\textbf{Proof of Theorem \ref{Thm2} for systems  $\mathtt{S}_1- \mathtt{S}_2$.}}

Now we shall prove that the discontinuous piecewise differential system $\eqref{syst:S1}-\eqref{syst:S2}$ separated by the straight line $\Sigma:x=0$, having three limit cycles. 

In $\Sigma^+$, we consider  the cubic isochronous center of type \eqref{syst:S1}
\begin{equation}\label{sist1-S1-S2}
    \begin{aligned}
    \dot{x} =& \frac{1}{100} \big( 9 x^3-12 x^2 (y+3)+x (y (3 y+46)+227)-4 (y (2 y+51)+5) \big), \\[4pt]
    \dot{y} =& \frac{1}{100} \big( 9 x^2 (y-3)-12 x ((y-2) y-28)+y (y (3 y+5)-247)-185 \big),
\end{aligned}
\end{equation}
with the first integral
\begin{equation*}
H_1(x,y) = \frac{\left(\frac{3}{10} x - \frac{1}{5} y - \frac{1}{5}\right)^2 + \left(\frac{1}{10} y - \frac{3}{10}\right)^2}{1 + 2 \left(\frac{1}{10} y - \frac{3}{10}\right) \left(\frac{3}{10} x - \frac{1}{5} y - \frac{1}{5}\right)}.
\end{equation*}

In $\Sigma^{-}$, we consider the cubic isochronous center of type \eqref{syst:S2}
\begin{equation}\label{sist2-S1-S2}
    \begin{aligned}
    \dot{x} =& \frac{1}{200} \big( -2872 x^3+x^2 (5112-6660 y)-2 x (111 y (23 y-32)+869) - \\
    & 37 y (y (35 y-66)+29)-46 \big), \\[4pt]
    \dot{y} =& \frac{1}{100} \big( 6 (595 x-254) y^2+(204 x (23 x-22)+581) y+34 x (12 x (5 x-8)+29) \\
    & +899 y^3+112 \big),
   \end{aligned}
\end{equation}
with first integral    
\begin{equation*}
H_2(x,y) = \frac{\left(\left(-\frac{3}{5} x - \frac{1}{2} y + \frac{7}{10}\right)^2+\left(- x - \frac{7}{10} y - \frac{1}{10}\right)^2\right)^2}{1 + 4 \left(-\frac{3}{5} x - \frac{1}{2} y + \frac{7}{10}\right) \left(- x -\frac{7}{10} y - \frac{1}{10}\right)}.
\end{equation*}

For the piecewise differential systems~\eqref{sist1-S1-S2}--\eqref{sist2-S1-S2}, crossing limit cycles intersect the discontinuity curve $\Sigma$ at pairs of distinct points $(p,q)$, with $p=(0,x)$ and $q=(0,y)$, $x \neq y$, if and only if
\begin{equation}\label{thm2:S1-S2}
    \begin{aligned}
        H_{1}(0,x) &= H_{1}(0,y), \\
        H_{2}(0,x) &= H_{2}(0,y). 
    \end{aligned}
\end{equation}

This is equivalent to
$$10 + 51y + x(51 + 4y)=0$$
\begin{align*}
&2300 + 25537y - 37296y^2 + 24642y^3  + 1369\,x^3(18 - 44y + 35y^2)  + x(25537 - 104192y \\
&\quad + 115810y^2 - 60236y^3) + 37\,x^2(-1008 + 3130y - 3588y^2 + 1295y^3)=0.
\end{align*}

Solving~\eqref{thm2:S1-S2}  yields three distinct real pairs $(p_i,q_i)$, where $p_i=(0,x_i)$ and $q_i=(0,y_i)$ for $i=1,2,3$, satisfying $x_i<y_i$. More precisely,
\begin{align*}
p_1 &= (0, -1.93773\ldots), \quad   q_1 = (0, 2.05378\ldots), \\
p_2 &= (0, -1.57213\ldots), \quad   q_2 = (0, 1.56959\ldots), \\
p_3 &= (0, -0.433316\ldots), \quad  q_3 = (0, 0.245584\ldots).
\end{align*}
These pairs characterize the crossing limit cycles depicted in Figure~\ref{fig-S1-S2}.

\begin{figure} 
\centering
\includegraphics[scale=0.30]{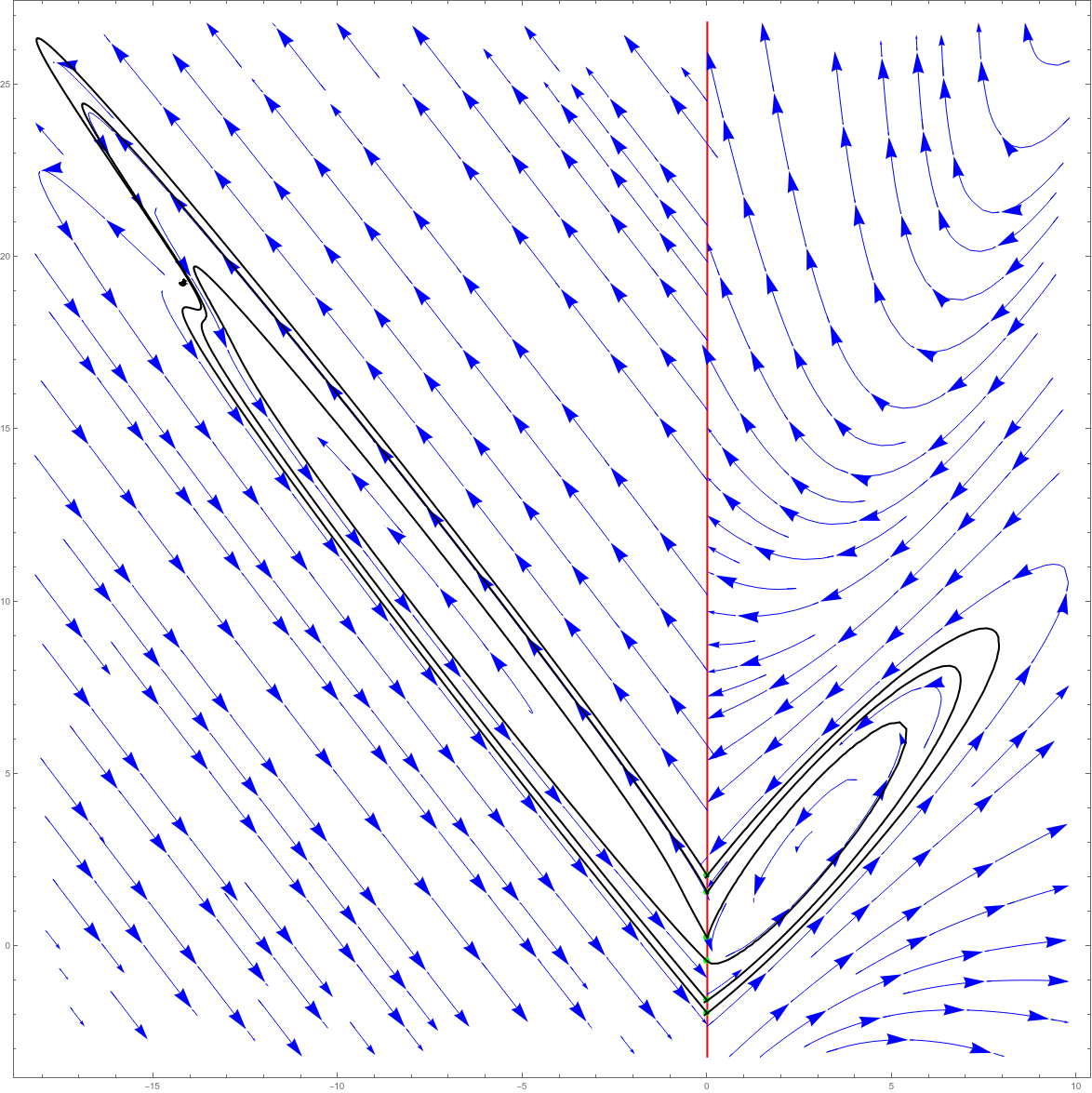}
\caption{The three limit cycle of the discontinuous piecewise differential system \eqref{sist1-S1-S2}-\eqref{sist2-S1-S2} of Theorem \ref{Thm2}.}
\label{fig-S1-S2}
\end{figure}

\subsection*{\textbf{Proof of Theorem \ref{Thm2} for systems  $\mathtt{S}_3- \mathtt{S}_3$.}}

\smallskip

Now we shall prove that the discontinuous piecewise differential system $\eqref{syst:S3}-\eqref{syst:S3}$ separated by the straight line $\Sigma:x=0$, having three limit cycles. 

In $\Sigma^+$,we consider the cubic isochronous center of type \eqref{syst:S3}
\begin{equation}\label{sist1-S3-S3}
    \begin{aligned}    
    \dot{x} =& \frac{1}{300} (5 x+5 y+3) \big(31 x^2-20 x (5 y+3)-10 y (5 y+6)+82 \big), \\[4pt]
    \dot{y} =& \frac{1}{1500} \big(-100 x^3+615 x^2 (5 y+3)+x (885 y (5 y+6)-1807)+10 (5 y+3) (5 y (5 y+6)-41) \big),   
 \end{aligned}
\end{equation}
with the first integral
\begin{equation*}
\begin{aligned}
\widetilde{H_3}(x,y) = \frac{\frac{9}{100} x^2 + \left(\frac{1}{2} x + \frac{1}{2} y + \frac{3}{10}\right)^2 - 4 \left(\frac{1}{2} x + \frac{1}{2} y + \frac{3}{10}\right)^4 + 4 \left(\frac{1}{2} x + \frac{1}{2} y + \frac{3}{10}\right)^6}{\left( - 1 + 3 \left(\frac{1}{2} x + \frac{1}{2} y + \frac{3}{10}\right)^2 \right)^3}.
\end{aligned}
\end{equation*}

In $\Sigma^{-}$, we consider the cubic isochronous center of type \eqref{syst:S3}
\begin{equation}\label{sist2-S3-S3}
    \begin{aligned}
    \dot{x} =& \frac{1}{1300} \big(-5008 x^3 + 6 x^2 (905 y+2323) + 6 x (5 y (105 y-404)-2021) - \\
    & 25 y (y (55 y+93)-211) + 3457 \big), \\[4pt]
    \dot{y} =& \frac{1}{3250} \big( 7552 x^3 - 24 x^2 (1035 y+737) + 5 x (3 y (3362-125 y)+3349) + \\
    & 5 y (5 y (415 y-174)-5169) - 6392 \big),      
   \end{aligned}
\end{equation}
with first integral    
\begin{equation*}
    \begin{aligned}
\widetilde{H_3}(x,y) =& \frac{ \left(\frac{4}{5} x + \frac{1}{2} y - \frac{9}{10}\right)^2 + \left(- x + y + \frac{4}{5}\right)^2 - 4 \left(\frac{4}{5} x + \frac{1}{2} y - \frac{9}{10}\right)^4 + 4 \left(\frac{4}{5} x + \frac{1}{2} y - \frac{9}{10}\right)^6 }{\left( - 1 + 3 \left(\frac{4}{5} x + \frac{1}{2} y - \frac{9}{10}\right)^2 \right)^3}.
   \end{aligned}
\end{equation*}

For the piecewise differential systems~\eqref{sist1-S3-S3}--\eqref{sist2-S3-S3}, crossing limit cycles intersect the discontinuity curve $\Sigma$ at pairs of distinct points $(p,q)$, with $p=(0,x)$ and $q=(0,y)$, $x \neq y$, if and only if
\begin{equation}\label{thm2:S3-S3}
    \begin{aligned}
        H_{3}(0,x) &= H_{3}(0,y), \\
    \widetilde{H_3}(0,x) &= \widetilde{H_3}(0,y).
    \end{aligned}
\end{equation}

This is equivalent to
\begin{align*}
&-436978 
+ 1500\,x^3(-11 + 15y)(29 + 15y) 
+ 625\,x^4(-11 + 15y)(29 + 15y) \\
&\quad - 1595\,y(6 + 5y)(-73 + 5y(6 + 5y)) + 30\,x(-11 + 15y)(29 + 15y)(-73 + 5y(6 + 5y)) \\
&\quad + 25\,x^2(-11 + 15y)(29 + 15y)(-37 + 5y(6 + 5y))=0.
\end{align*}
\begin{align*}
&x\Big(-282768772 
+ 5x\big(63632426 
+ 34975x(-2196 + 5x(223 + x(-54 + 5x)))\big)\Big) \\
&\quad + 282768772\,y + 450x^2\Big(3704272 
+ 1425x(-2196 + 5x(223 + x(-54 + 5x)))\Big)y \\
&\quad + 55\Big(-5784766 
+ 45x\big(-673504 
+ 125x^2(-2196 + 5x(223 + x(-54 + 5x)))\big)\Big) y^2 \\
&\quad + 274500\,(1399 + 45x(114 + 55x))\,y^3  - 139375\,(1399 + 45x(114 + 55x))\,y^4 \\
&\quad + 33750\,(1399 + 45x(114 + 55x))\,y^5  - 3125\,(1399 + 45x(114 + 55x))\,y^6=0.
\end{align*}
Solving~\eqref{thm2:S3-S3}  yields three distinct real pairs $(p_i,q_i)$, where $p_i=(0,x_i)$ and $q_i=(0,y_i)$ for $i=1,2,3$, satisfying $x_i<y_i$. More precisely,
\begin{align*}
p_1 =& (0, 0.733638 \ldots), \quad q_1=(0, 2.79455 \ldots), \\
p_2 =& (0, -5.33051 \ldots), \quad q_2=(0, -1.93341 \ldots), \\
p_3 =& (0, -0.762889 \ldots), \quad q_3=(0, -0.437111 \ldots),
\end{align*}

These pairs characterize the crossing limit cycles depicted in Figure \ref{fig-S3-S3}. 

\begin{figure}
\centering
\includegraphics[scale=0.30]{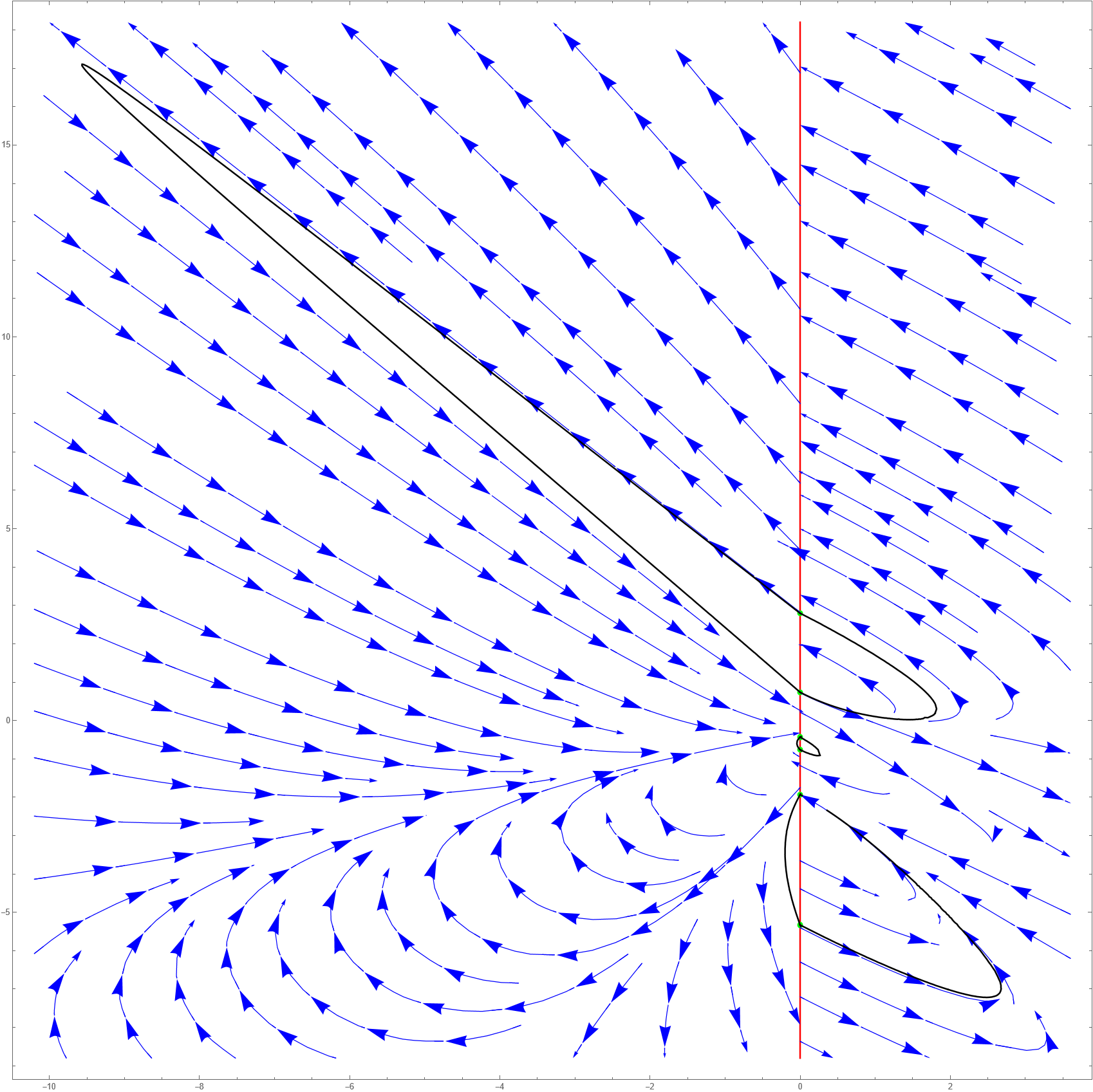}
\caption{The three limit cycle of the discontinuous piecewise differential system \eqref{sist1-S3-S3}-\eqref{sist2-S3-S3} of Theorem \ref{Thm2}.}
\label{fig-S3-S3} 
\end{figure}

\subsection*{\textbf{Proof of Theorem \ref{Thm2} for systems  $\mathtt{S}_3- \mathtt{S}_4$.}}

\smallskip

Now we shall prove that the discontinuous piecewise differential system $\eqref{syst:S3}-\eqref{syst:S4}$ separated by the straight line $\Sigma:x=0$, having three limit cycles. 

In $\Sigma^+$,  we consider the cubic isochronous center of type \eqref{syst:S3}
\begin{equation}\label{sist1-S3-S4}
    \begin{aligned}
    \dot{x} =& \frac{1}{700} \big( 1944 x^3 + 324 x^2 (2 y-13) - 24 x y (46 y+95) + 2558 x + 2 y (4 y (81-16 y) + 433) \\
    & - 605 \big), \\[4pt]
    \dot{y} =& \frac{1}{1400} \big( 2808 x^3 + 36 x^2 (124 y-113) - 528 x y (5 y+19) + 2830 x + 4 y (66 y (1-4 y) + 1073) \\
    & - 1157 \big),       
   \end{aligned}
\end{equation}
with the first integral
\begin{equation*}
H_3(x,y) = \frac{ \left(-\frac{3}{5} x + \frac{2}{5} y + \frac{7}{10}\right)^2 + \left(-\frac{4}{5} x - \frac{2}{5} y + \frac{2}{5}\right)^2 - 4 \left(-\frac{3}{5} x + \frac{2}{5} y + \frac{7}{10}\right)^4  + 4 \left(-\frac{3}{5} x + \frac{2}{5} y + \frac{7}{10}\right)^6}{\left(- 1 + 3 \left(-\frac{3}{5} x + \frac{2}{5} y + \frac{7}{10}\right)^2\right)^3}.
\end{equation*}

In $\Sigma^{-}$, we consider the cubic isochronous center of type \eqref{syst:S4}
\begin{equation}\label{sist2-S3-S4}
    \begin{aligned}
    \dot{x} =& -0.621976 x^3 + x^2 (-2.24036 y-1.20363) + x (y (5.88238 y+18.0032)+13.1839) \\
    & + y ((-2.13752 y-12.6274) y-24.6751) - 15.8508, \\[4pt]
    \dot{y} =& -0.0474434 x^3 + x^2 (-2.57852 y-3.02589) + x (y (4.60748 y+15.4492)+13.1959) \\
    & + y ((-1.2664 y-8.39959) y-17.7854) - 12.5349,          
   \end{aligned}
\end{equation}
with first integral    
\begin{equation*}
    \begin{aligned}
H_4(x,y) =& \frac{1}{\left( 1 + 3 \left(\frac{5}{9} x - \frac{7}{9} y - 1.40238\right)^2\right)^3} \Big( \left(\frac{5}{9} x - \frac{7}{9} y - 1.40238\right)^2 \\
& + 4 \left(\frac{5}{9} x - \frac{7}{9} y - 1.40238\right)^4 + 4 \left(\frac{5}{9} x - \frac{7}{9} y - 1.40238\right)^6 \\
& + \left(\frac{1}{3} x + 0.111599 y - 0.136717\right)^2  \Big).
   \end{aligned}
\end{equation*}

For the piecewise differential systems~\eqref{sist1-S3-S4}--\eqref{sist2-S3-S4}, crossing limit cycles intersect the discontinuity curve $\Sigma$ at pairs of distinct points $(p,q)$, with $p=(0,x)$ and $q=(0,y)$, $x \neq y$, if and only if
\begin{equation}\label{thm2:S3-S4}
    \begin{aligned}
        H_{3}(0,x) &= H_{3}(0,y),\\ 
        H_{4}(0,x) &= H_{4}(0,y).
    \end{aligned}
\end{equation}
This is equivalent to
\begin{align*}
&16240 x^3(391 + 72y(-19 + 4y)) 
+ 10160 x^4(391 + 72y(-19 + 4y)) \\
&\quad + 2688 x^5(391 + 72y(-19 + 4y)) 
+ 256 x^6(391 + 72y(-19 + 4y)) \\
&\quad + x\Big(1336445 
+ 72y^2\big(199327 + 304y(1015 + y(635 + 8y(21 + 2y)))\big)\Big) \\
&\quad - y\Big(1336445 
+ y\big(3820583 + 6256y(1015 + y(635 + 8y(21 + 2y)))\big)\Big) \\
&\quad - x^2\Big(-3820583 
+ 72y\big(199327 + 64y^2(1015 + y(635 + 8y(21 + 2y)))\big)\Big)=0.
\end{align*}
\begin{align*}
&-521640 x^3(-887 + 108y(-15 + 2y)) 
+ 270540 x^4(-887 + 108y(-15 + 2y)) \\
&\quad - 81648 x^5(-887 + 108y(-15 + 2y)) 
+ 11664 x^6(-887 + 108y(-15 + 2y)) \\
&\quad + 2x\Big(302421613 
+ 54y^2\big(8426899 + 4860y(-1610 + y(835 + 36(-7 + y)y))\big)\Big) \\
&\quad + y\Big(-604843226 
+ y\big(578956391 + 287388y(-1610 + y(835 + 36(-7 + y)y))\big)\Big) \\
&\quad - x^2\Big(578956391 
+ 108y\big(8426899 + 648y^2(-1610 + y(835 + 36(-7 + y)y))\big)\Big)=0.
\end{align*}
Solving~\eqref{thm2:S3-S4} yields three distinct real pairs $(p_i,q_i)$, where $p_i=(0,x_i)$ and $q_i=(0,y_i)$ for $i=1,2,3$, satisfying $x_i<y_i$. More precisely,
\begin{align*}
p_1 =& (0, -2.73556 \ldots), \quad q_1=(0, -0.551073 \ldots), \\
p_2 =& (0, -2.3282 \ldots), \quad q_2=(0, -0.815007 \ldots), \\
p_3 =& (0, -1.78878 \ldots), \quad q_3=(0, -1.2422 \ldots),
\end{align*}
These pairs characterize the crossing limit cycles depicted in Figure \ref{fig-S3-S4}. 

\begin{figure}
\centering
\includegraphics[scale=0.30]{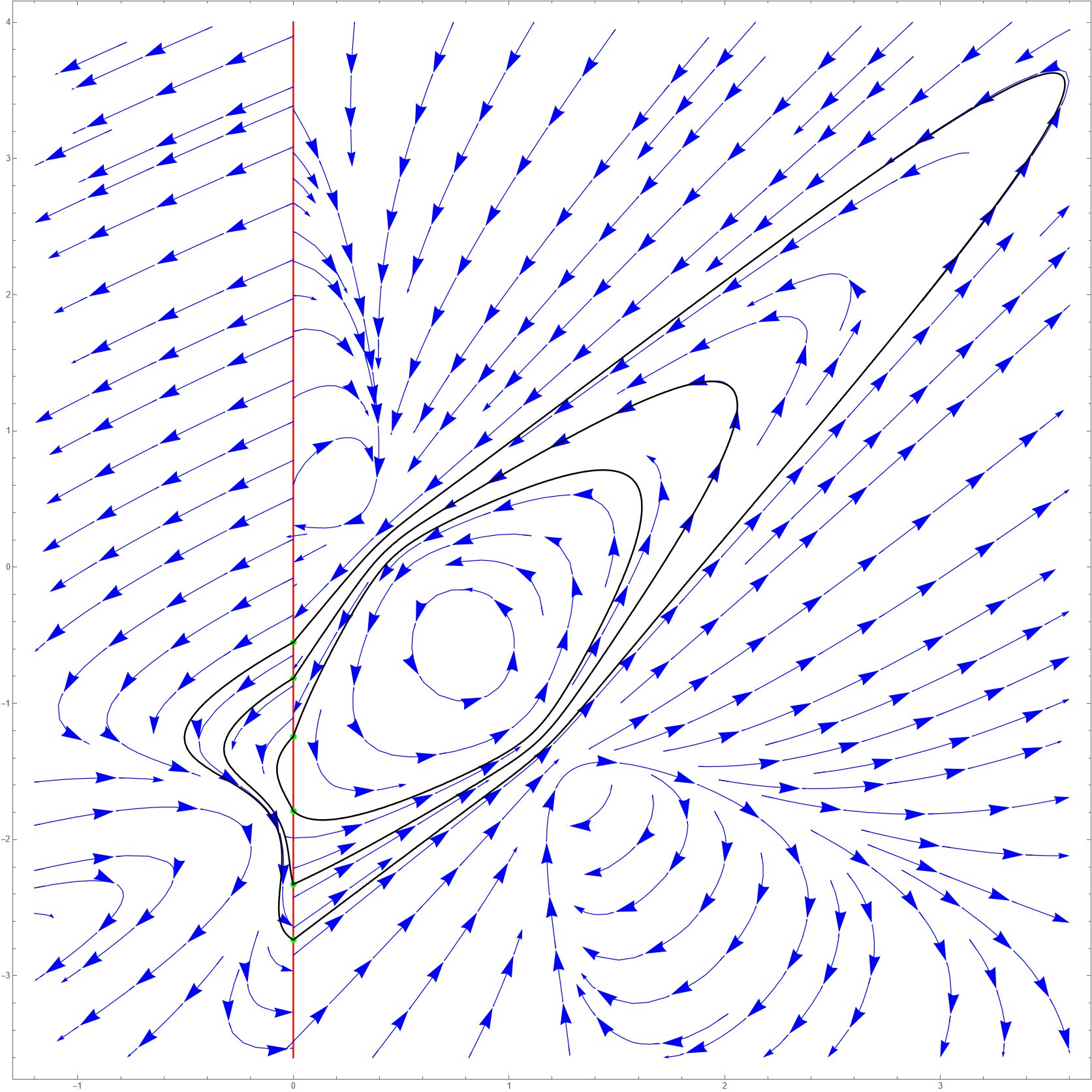}
\caption{The three limit cycle of the discontinuous piecewise differential system \eqref{sist1-S3-S4}-\eqref{sist2-S3-S4} of Theorem \ref{Thm2}.}
\label{fig-S3-S4} 
\end{figure}

\subsection*{\textbf{Proof of Theorem \ref{Thm2} for systems  $\mathtt{S}_4- \mathtt{S}_4$.}}

\smallskip

Now we shall prove that the discontinuous piecewise differential system $\eqref{syst:S4}-\eqref{syst:S4}$ separated by the straight line $\Sigma:x=0$, having three limit cycles. 

In $\Sigma^+$, we consider the cubic isochronous center of type \eqref{syst:S4}
\begin{equation}\label{sist1-S4-S4}
\begin{aligned}
\dot{x} &= \frac{1}{4900}\Big(
24066 - 3086 x^3 + x^2 (28203 - 9687 y)
- 2 y \big(-22646 + 5 y (2037 + 1090 y)\big) \\
&\qquad + x \big(-55945 + 3 y (-4774 + 8845 y)\big)
\Big), \\[6pt]
\dot{y} &= \frac{1}{4900}\Big(
12131 + 382 x^3
- 6 x^2 (357 + 25 y)
+ x \big(-11057 + 3 (5894 - 2847 y) y\big) \\
&\qquad + y \big(-1679 + y (-9303 + 8795 y)\big)
\Big),
\end{aligned}
\end{equation}
with the first integral
\begin{equation*}
\begin{aligned}
H_4(x,y) = \frac{
\left(-\frac{7}{10} + \frac{4x}{5} - \frac{9y}{10}\right)^2
+ \left(-\frac{7}{10} + \frac{x}{10} + \frac{y}{2}\right)^2
+ 4\left(-\frac{7}{10} + \frac{x}{10} + \frac{y}{2}\right)^4
+ 4\left(-\frac{7}{10} + \frac{x}{10} + \frac{y}{2}\right)^6
}{
\left(1 + 3\left(-\frac{7}{10} + \frac{x}{10} + \frac{y}{2}\right)^2\right)^3
}.
\end{aligned}
\end{equation*}

In $\Sigma^-$, we consider the cubic isochronous center of type \eqref{syst:S4}
\begin{equation}\label{sist2-S4-S4}
\begin{aligned}
\dot{x} &= \frac{
863 + 3 x \big(209 + 72 x (8 + 17 x)\big)
+ 851 y + 18 x (-74 + 171 x) y
- 72 (9 + 8 x) y^2
- 286 y^3
}{1500}, \\[6pt]
\dot{y} &= \frac{1}{500}\Big(
-537 + 3 x \big(-331 + 18 x (37 + 23 x)\big)
+ (-1409 + 72 x (4 + 39 x)) y\\
&\qquad + 6 (-203 + 69 x) y^2
- 496 y^3
\Big),
\end{aligned}
\end{equation}
with the first integral
\begin{equation*}
\widetilde{H}_4(x,y) =\frac{
\left(-\frac{1}{2} - \frac{9x}{10} - \frac{7y}{10}\right)^2
+ \left(-\frac{1}{5} + \frac{3x}{5} - \frac{y}{5}\right)^2
+ 4\left(-\frac{1}{5} + \frac{3x}{5} - \frac{y}{5}\right)^4
+ 4\left(-\frac{1}{5} + \frac{3x}{5} - \frac{y}{5}\right)^6
}{
\left(1 + 3\left(-\frac{1}{5} + \frac{3x}{5} - \frac{y}{5}\right)^2\right)^3
}.
\end{equation*}

For the piecewise differential systems~\eqref{sist1-S4-S4}--\eqref{sist2-S4-S4}, crossing limit cycles intersect the discontinuity curve $\Sigma$ at pairs of distinct points $(p,q)$, with $p=(0,x)$ and $q=(0,y)$, $x \neq y$, if and only if
\begin{equation}\label{thm2:S4-S4}
    \begin{aligned}
        H_{4}(0,x) &= H_{4}(0,y), \\
    \widetilde{H_4}(0,x) &= \widetilde{H_4}(0,y).
    \end{aligned}
\end{equation}

This is equivalent to
\begin{align*}
&x\Big(-1468242594 
+ x\big(-133306394 
+ 150625\, x \big(-1932 + 5x\big(167 + x(-42 + 5x)\big)\big)\big)\Big) \\
&\quad + 126\Big(11652719 
+ 5x^2\big(7716527 
+ 2000x\big(-1932 + 5x\big(167 + x(-42 + 5x)\big)\big)\big)\Big) y \\
&\quad + \Big(133306394 
- 4861412010\, x 
+ 613125\, x^3 \big(-1932 + 5x\big(167 + x(-42 + 5x)\big)\big)\Big) y^2 \\
&\quad + 1207500\,(241 + 9x(224 + 109x))\, y^3  - 521875\,(241 + 9x(224 + 109x))\, y^4 \\
&\quad + 131250\,(241 + 9x(224 + 109x))\, y^5  - 15625\,(241 + 9x(224 + 109x))\, y^6=0.
\end{align*}
\begin{align*}
&120x^3(239 + 9y(202 + 143y)) 
+ 40x^4(239 + 9y(202 + 143y))  + 6x^5(239 + 9y(202 + 143y)) 
\\
&\quad + x^6(239 + 9y(202 + 143y)) - 2x\Big(676592 + 9y^2\big(336 + 101y(120 + y(40 + y(6 + y)))\big)\Big) \\
&\quad - y\Big(-1353184 + y\big(-957152 + 239y(120 + y(40 + y(6 + y)))\big)\Big) \\
&\quad - x^2\Big(957152 + 9y\big(-672 + 143y^2(120 + y(40 + y(6 + y)))\big)\Big)=0.
\end{align*}
Solving~\eqref{thm2:S4-S4} yields three distinct real pairs $(p_i,q_i)$, where $p_i=(0,x_i)$ and $q_i=(0,y_i)$ for $i=1,2,3$, satisfying $x_i<y_i$. More precisely,
\begin{align*}
p_1 =& (0, -1.24884 \ldots), \quad q_1=(0, -0.166809 \ldots), \\
p_2 =& (0, -0.123252 \ldots), \quad q_2=(0, 11.2846 \ldots), \\
p_3 =& (0, 0.165367 \ldots), \quad q_3=(0, 4.39164 \ldots),
\end{align*}
These pairs characterize the crossing limit cycles depicted in Figure \ref{fig-S4-S4}. 

\begin{figure}
\centering
\includegraphics[scale=0.30]{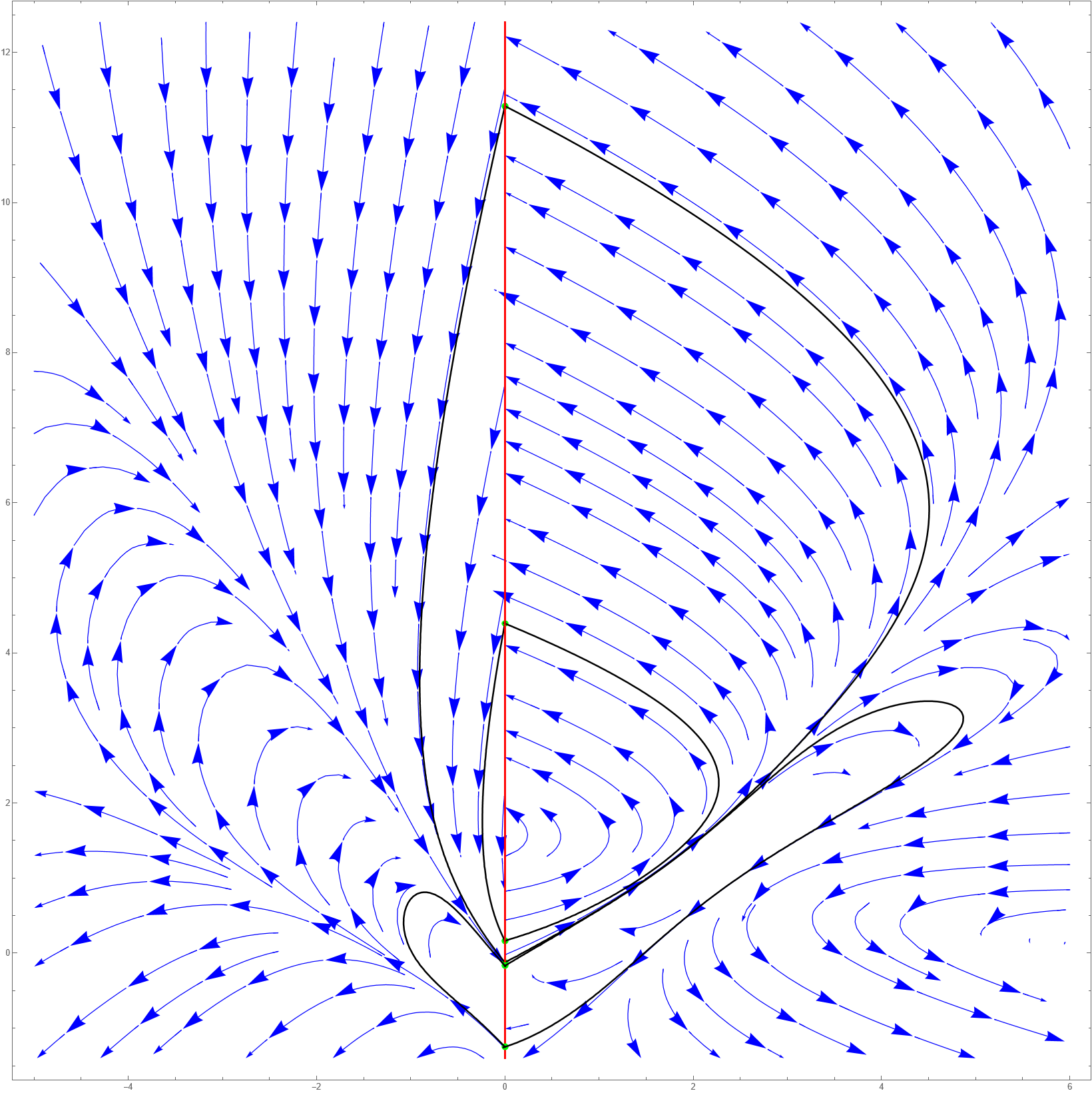}
\caption{The three limit cycle of the discontinuous piecewise differential system \eqref{sist1-S4-S4}-\eqref{sist2-S4-S4} of Theorem \ref{Thm2}.}
\label{fig-S4-S4} 
\end{figure}



It remains an open problem whether the upper bounds established in Theorem~\ref{Thm1} are sharp. In contrast, Theorem~\ref{Thm2} provides explicit constructions of systems exhibiting three crossing limit cycles, establishing a nontrivial lower bound within this class. The gap between the current upper and lower bounds remains unresolved.

\section*{Acknowlegements}

The first author acknowledges partial support from the 2026 Summer Postdoctoral Program at the Instituto de Matemática Pura e Aplicada (IMPA) and from a scholarship granted by the Fundação Arthur Bernardes (FUNARBE).
The second author acknowledges partial support from CNPq under grant No.~169201/2023-6.


%
\bibliographystyle{acm}
\bibliography{sample}

\end{document}